\documentclass{amsart}
\usepackage{amsmath}
\usepackage{amsfonts, amssymb}
\usepackage{amsthm, upref}
\usepackage{graphicx}
\usepackage{enumerate}
\usepackage[usenames, dvipsnames]{color}

\usepackage{ifthen}
\usepackage{tikz}
\usepackage{appendix}
\usepackage{verbatim}
\usetikzlibrary{decorations.pathmorphing}
\usetikzlibrary{calc}
\usepackage{hyperref}

\input xy
\xyoption{all}

\tikzset{vert/.style={circle,fill,inner sep=0,
    minimum size=0.15cm,draw}, new/.style={}}

\renewcommand{\comment}[1]{}
\newcommand{\eq}{\begin{equation}}
\newcommand{\en}{\end{equation}}
\newcommand{\rr}{\mathbb{R}}
\newcommand{\NN}{\mathbb{N}}

\newcommand{\abs}[1]{\left\lvert #1 \right\rvert}

\newcommand{\mcal}[1]{\mathcal{#1}}

\newcommand{\nin}{\noindent}
\newcommand{\tbf}{\textbf}
\newcommand{\Exp}{\text{Exp}}
\newcommand{\Ww}{\mathcal{W}}

\newcommand{\bad}{\text{birth-and-death }}
\newcommand{\E}{\mathrm{E}}
\renewcommand{\P}{\mathrm{P}}
\newcommand{\leb}{\mathrm{Leb}}
\newcommand{\Cov}{\mathrm{Cov}}
\newcommand{\cov}{\mathrm{Cov}}
\newcommand{\Var}{\mathrm{Var}}
\newcommand{\stime}{\overline{\sigma}}
\newcommand{\tchi}{\widetilde{\chi}}

\newcommand{\bG}[2]{\overleftarrow{G}_{#1}(#2)}

\newcommand{\bX}[4][s]{\overleftarrow{M}_{#2\ifthenelse{\equal{#3}{}}{}{,#3}}^{(#1)}(#4)}
\newcommand{\bGraph}[2]
{\ifthenelse{\equal{#1}{}}{\overleftarrow{\Gamma}(#2)}
{\overleftarrow{\Gamma}_{#1}(#2)}}
\newcommand{\gevent}{\mathcal{E}}
  \newcommand{\Cy}[2][\infty]{C_{#2}^{(#1)}}

\newcommand{\e}{\varepsilon}

\newcommand{\tr}{\mathrm{tr}}
\newcommand{\hQ}{\widehat{Q}}
\newcommand{\hq}{\widehat{q}}
\newcommand{\CNBW}[2][\infty]{\mathrm{CNBW}_{#2}^{(#1)}}
\newcommand{\divides}{\mid}

\begin{document}

\theoremstyle{plain}
\newtheorem{thm}{Theorem}
\newtheorem{lemma}[thm]{Lemma}
\newtheorem{prop}[thm]{Proposition}
\newtheorem{cor}[thm]{Corollary}

\theoremstyle{definition}
\newtheorem{defn}{Definition}
\newtheorem{asmp}{Assumption}
\newtheorem{notn}{Notation}
\newtheorem{prb}{Problem}

\theoremstyle{remark}
\newtheorem{rmk}{Remark}
\newtheorem*{rmk*}{Remark}
\newtheorem{exm}{Example}
\newtheorem{clm}{Claim}

\title[random transposition dynamics]{The random transposition dynamics on random regular graphs and the Gaussian free field}

\author{Shirshendu Ganguly and Soumik Pal}
\address{Department of Mathematics\\ University of Washington\\ Seattle, WA 98195}
\email{sganguly@math.washington.edu, soumikpal@gmail.com}



\thanks{Soumik's research is partially supported by NSF grant DMS-1308340}

\date{\today}

\begin{abstract} 
A single permutation, seen as union of disjoint cycles, represents a regular graph of degree two. Consider $d$ many independent random permutations and superimpose their graph structures. It is a common model of a random regular (multi-) graph of degree $2d$. We consider the following dynamics. The dimension (i.e. size) of each permutation grows by coupled Chinese Restaurant Processes, while in ÔtimeÕ each permutation evolves according to the
random transposition chain. Asymptotically in the size of the graph one observes a remarkable evolution of short cycles and linear eigenvalue statistics in dimension and time. In dimension, it was shown by Johnson and Pal \cite{JP14} that cycle counts are described by a Poisson field of Yule processes. Here, we give a Poisson random surface description in dimension and time of the limiting cycle counts for every $d$. As $d$ grows to infinity, the fluctuation of the limiting cycle counts, across dimension, converges to the Gaussian Free Field. In time this field is preserved by a stationary Gaussian dynamics. The laws of these processes are similar to eigenvalue fluctuations of the minor process of a real symmetric Wigner matrix whose coordinates evolve as i.i.d. stationary stochastic processes.
\end{abstract}

\maketitle

\section{Introduction} 

We begin with a heuristic description of our model and the main results. Precise formulations are given in the following subsection. Consider $d$ independent random permutations on $n$ labels $[n]:=\left\{ 1,2,\ldots,n \right\}$. Every permutation has a corresponding permutation matrix whose entries are zero or one. We add all the $d$ matrices and further add it to its own transpose. This produces a symmetric matrix such that the sum of entries in every row is exactly $2d$. The matrix can be thought of as the adjacency matrix of a regular (multi)graph that allows loops and multiple edges. The model of the random graph generated by this procedure  is called the permutation model of a random regular graph of degree $2d$ on $n$ vertices. We will denote this random graph by $G(n,2d)$. For more on the recent uses and applications of the permutation model see \cite{friedmane2, friedmanalon}. 

Our objective in this paper is to consider stochastic processes of such random regular graphs indexed by two parameters: size (or `order' or `dimension') and time. The inspiration comes from a standard model of random matrix theory: a Gaussian Wigner matrix called the Gaussian orthogonal ensemble (GOE). In that model, at a given time and dimension $n$, one has an $n\times n$ symmetric matrix of independent mean-zero Gaussian upper triangular entries (all entries have variance one, except the diagonal elements which have variance two). The matrix grows to dimension $(n+1)\times (n+1)$ by adding an independent $(n+1)$th row of independent Gaussian entries and, hence, a column, by symmetry. This gives us a matrix-valued process indexed by $n$ (called the minor process). If, now, the individual entries are replaced by independent Gaussian stochastic processes (say, stationary Ornstein-Uhlenbeck diffusions), then we have a field of random matrices indexed by size and time that is known to display remarkable properties.

The corresponding field of random regular graphs is non-trivial since the entries of the adjacency matrix are not independent. However, each graph is constructed using random permutations. One can borrow well-known dynamics on every random permutation: the Chinese Restaurant Process (CRP) to grow its size and the random transposition Markov chain to evolve it in time. Dynamics on individual permutations must be coupled to ensure that the same set of labeled vertices are preserved in dimension and time.

To do this, imagine observing an infinite sequence $\{\pi_1,\pi_2,\ldots \}$ of independent random permutations of $\NN$. By a `random permutation' of $\NN$ we mean the following.
The $i$th element $\pi_i$ is a sequence of permutations visualized as a tower. The $n$th level of this tower is a random permutation of $[n]$ that grows in $n$ according to the CRP. This, in turn, produces an array of permutations matrices doubly indexed by $(i,n)$. Consider the sequence of partial sums of these permutation matrices along $i$ for every $n$. Symmetrize every matrix by adding it to its transpose. Consider the $d$th partial sum and consider the process of matrices growing in $n$. The corresponding sequence of graphs produces a coupling of $G(n,2d)$ for all values of $n$.  


We now describe the evolution in time. Attach each element of $\NN$ with a sequence of i{.}i{.}d{.} exponential clocks. When any clock rings the corresponding element chooses a `uniform' element from $\NN$ and every permutation tower gets multiplied on the left by the transposition of the pair. Clearly this statement as it is does not make sense since there are countably infinitely many elements. However, as we show later, there is a way to make this precise. The evolving family of towers of permutations now produces a family $G(n,2d,s)$ of regular graphs indexed by the triplet (order, degree, time). Consider the following graph statistics (i) counts of cycles of a fixed size, (ii) polynomial linear eigenvalue statistics. We study their functional limits for large order and suitably scaled time both for  fixed $d$ and as $d$ tends to infinity.  

For very large order, the cycle counts form an approximate polynomial basis for the linear eigenvalue statistics which makes (i) and (ii) asymptotically equivalent. We provide a precise process description of the joint evolution of cycles of various sizes across dimension (i.e., order) and time. Informally, suppose $\left( N_k(t,s),\; k \in \NN \right)$ denote the number of $k$ cycles in the graph for \textit{very large} order $t$ and \textit{very small} time $s$, then its joint law is approximately that of countably many random Poisson surfaces (one per $k$) given by Yule processes in dimension and approximate \bad chains in time. The dimension part of it is described in \cite{JP14}. As $d$ goes to infinity, the fluctuation of the above infinite dimensional surface converges to a product of countably many two dimensional Gaussian random surfaces (one for each $k$) each of which has stationary one dimensional marginals. The covariances of these random surfaces coincide with that of Chebyshev polynomial eigenvalue statistics of Wigner matrices, growing in dimension as the minor process, while each entry (up to symmetry) moves in time as independent processes. See \cite{Bor2, Bor1}. It follows that the fluctuation of the height function of eigenvalue distribution (for every fixed time) is distributed approximately as the Gaussian Free Field (GFF) on the upper half plane and zero boundary condition, whose law is kept preserved in time  by the random transposition Markov chain.

\subsection{Formal description of the model} A part of the description of our model already appears in  \cite{JP14} and \cite{Johnson14}, where the reader can find more references on the subject. Consider a permutation $\pi^{(n)}$ on the $n$ labels $[n]:=\{ 1,2, \ldots, n \}$. Consider the permutation matrix corresponding to $\pi^{(n)}$, add it to its transpose. The resulting matrix can be thought of as the adjacency matrix of a vertex labeled $2$-regular graph that allows multiple edges and loops.

We will now grow this graph in dimension and transform it in time. In dimension this will be done by the Chinese restaurant Process (CRP). By a tower of random permutations we mean a sequence of random permutations $(\pi^{(n)}, n\in \NN)$ such that
\begin{enumerate}
\item[(i)] $\pi^{(n)}$ is a uniformly distributed random permutation of $[n]$ for each $n$, and
\item[(ii)] for each $n$, if $\pi^{(n)}$ is written as a product of cycles then $\pi^{(n-1)}$ is derived from $\pi^{(n)}$ by deletion of the element $n$ from its cycle. 
\end{enumerate}
The CRP is a Markov chain that reverses the above procedure by building $\pi^{(n)}$ from $\pi^{(n-1)}$. See  \cite[Section 3.1]{Pit}. 
 
Now suppose we construct a countable collection $\{ \Pi_d,\, d \in \NN \}$ of towers of random permutations. We will denote the permutations in $\Pi_d$ by $\left\{ \pi_d^{(n)},\, n \in \NN\right\}$. Then it is possible to model every possible $G(n, 2d)$ by adding the permutation matrices (and their transposes) corresponding to $\left\{ \pi^{(n)}_j,\, 1\le j \le d \right\}$. In what follows we will keep $d$ fixed and consider $n$ as a growing parameter. Thus, $G_n$ will represent $G(n,2d)$ for some fixed $d$. Here and later, $G_0$ will represent the empty graph. 
We construct a continuous-time version of a graph-valued Markov chain by inserting new vertices
into $G_n$ with rate $n+1$.  Formally,
define independent \label{page:poissonization}
times $T_i\sim\Exp(i)$, exponential distribution with rate $i$, and let 
\begin{align}\label{eq:continuoustimeexp}
  M_t=\max\left\{m\colon\ \sum_{i=1}^mT_i\leq t\right\},
\end{align}
and consider the continuous-time Markov chain
$G_{M_t}$ for $t\in [0, \infty)$. We are now going to abuse our notation and define a two-parameter family of graphs $G(t,s)$, for nonnegative parameters $t$ and $s$. The second index $s$ in $G(t,s)$ refers to time, while the first index $t$ refers to dimension. This should not produce confusion with the notation $G(n,2d)$, where the parameters are integers and $d$ is fixed. We start by defining $G(t,0)=G_{M_t}$ for all $t\ge 0$. 

We now describe the movement in time. Fix some positive $T$ and suppose $M_T=n$. Let $\tau_{ij}$ be the transposition $(i,j)$. Consider the finite set of transpositions of elements in $[n]$. Consider the Markov chain that independently chooses a uniform random transposition and multiplies to each permutation $\pi_i^{(n)}$ on the left. This is the well-known random transposition Markov chain which keeps the joint law of $d$ independent permutations $\left\{ \pi_i^{(n)},\; i\in[d]  \right\}$ invariant. 

As before, we will actually modify the above to a continuous time chain. Suppose we attach a sequence of i.i.d. exponential one clocks with every label in $[n]$. As the first clock rings, the corresponding element (say $I_1$) chooses a uniformly random element in $[n]$ except itself (say $J_1$) and we multiply every permutation $\left\{ \pi_j^{(n)},\; j\in[d]\right\}$ on the left by $\tau_{I_1J_1}$. With every successive ring one takes a successive product $\tau_{I_k J_k} \cdot \ldots \cdot \tau_{I_1 J_1} \cdot \pi^{(n)}_j$.

Thus every possible transposition occurs with rate $2/(n-1)$ and we successively multiply them on the left. After time $s$, let $\sigma_s$ denote the (left) product of successive transpositions so far. Then, each permutation is now modified to $\sigma_s \cdot \pi_j^{(n)}$. The graph induced by these permutations will be denoted by $G(T,s)$.

Given initial permutations $\left\{\pi^{(n)}_i,\; i \in [d]\right\}$ and $\sigma_s$ for some $s\ge 0$, we  define $G(t,s)$ for $t\in [0, T)$ by  successively removing elements in the order 
\eq\label{eq:orderremoval}
\sigma_s(n), \sigma_s(n-1), \ldots, \sigma_s(1).
\en
More precisely, recall the sequence of exponential times $T_1, T_2, \ldots, T_n$ that defined $G(T,0)$ in \eqref{eq:continuoustimeexp} with $M_T=n$. Consider the CRP backwards at time $0$ as it removes vertices from each permutation. When it removes vertex $i$ at time $0$, simultaneously remove vertex $\sigma_s(i)$ from at time $s$. This is the CRP running backwards on the relabeled vertices \eqref{eq:orderremoval}. Hence, the law of the unlabeled graph-valued process in dimension remains unchanged along time. Additionally, by the above coupling every $G(t,s)$, $s\ge 0$, has the same number of vertices for all $0\le t \le T$, although their labels might be different. 

\begin{rmk*}
Recall the comment made in the introduction about the difficulty in defining transposition Markov chain on a permutation on the entire $\NN$.
We get around this problem by defining the transposition chain for a permutation of a large dimension $T$ and then project back to smaller dimensions by running the CRP backwards. We will later take $T\to \infty$ to construct a substitute for the entire $\NN$. 
\end{rmk*}

For every $T >0$ the above construction produces a doubly-indexed family of graphs $\left\{ G(t,s),\; 0\le t \le T,\; s\ge 0  \right\}$, where $t$ represents dimension and $s$ represents time. We intend to study the asymptotic behavior of short cycles and linear eigenvalue statistics of this process as $T$ grows to infinity. For a fixed time, the asymptotic law of the process in dimension has been already studied in \cite{JP14}, to be described in the next section. The main focus of this work is to study the joint evolution of the graph statistics in dimension and time and to draw parallel with results in \cite{Bor2}. 

\subsection{Notation and Definitions} We are interested in the dynamics of cycles of a given size. However, not all cycles of a fixed size behave identically. To obtain a nice Markovian description, we need to classify cycles by the permutations that produce its edges. The following concepts are recalled from \cite{JP14}.

   \begin{figure}[t]
      \begin{center}
        \begin{tikzpicture}[scale=1.5,vert/.style={circle,fill,inner sep=0,
              minimum size=0.15cm,draw},>=stealth]
          \node[vert] (s0) at (270:1) {};
            \node[vert] (s1) at (330:1) {};
            \node[vert] (s2) at (30:1) {};
            \node[vert] (s3) at (90:1) {};
            \node[vert] (s4) at (150:1) {};
            \node[vert] (s5) at (210:1) {};
            \draw[thick,->] (s0) to node[auto,swap] {$\pi_2$} (s1);
            \draw[thick,<-] (s1) to node[auto,swap] {$\pi_1$} (s2);
            \draw[thick,->] (s2) to node[auto,swap] {$\pi_2$} (s3);
            \draw[thick,->] (s3) to node[auto,swap] {$\pi_1$} (s4);
            \draw[thick,->] (s4) to node[auto,swap] {$\pi_2$} (s5);
            \draw[thick,<-] (s5) to node[auto,swap] {$\pi_3$} (s0);
        \end{tikzpicture}
      \end{center}
    \caption{A cycle whose word
    is the equivalence class of
    $\pi_2\pi_1^{-1}\pi_2\pi_1\pi_2\pi_3^{-1}$ in $\Ww_6/D_{12}$. Here $h(w)=1$, $b(w)=2$, $c(w)=0$.}
    \label{fig:cycleword}
  \end{figure}
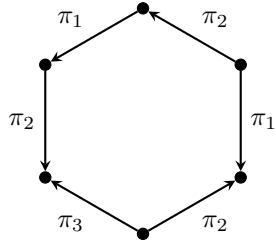

Imagine the graph $G_n$ as a directed, edge-labeled
  graph in a natural way.  For convenience, drop superscripts
  and let $\pi_l=\pi^{(n)}_l$.
  If $\pi_l(i)=j$, then we imagine this edge to be directed from $i$ to $j$ and to be
  labeled by $\pi_l$.

  Consider a walk on $G_n$ (i.e., a sequence of neighboring vertices) and write
  down the label of each edge as it is traversed, putting
  $\pi_i$ or $\pi_i^{-1}$ according to the direction
  we walk over the edge. Any such sequence of $\pi_i$ or $\pi^{-1}_i$ will be called a word. 
 
   We call a walk \emph{closed} if it starts and ends at the same vertex,
  and we call a closed walk a cycle if
  it never visits a vertex twice until the very last one, 
  and it never traverses an
  edge more than once in either direction.  
  Thus a word $w=w_1\cdots w_k$ that corresponds to traversing a cycle
  is cyclically reduced, i.e., $w_i\neq w_{i+1}^{-1}$ for all $i$, 
  considering $i$ modulo
  $k$. 

  Let $\Ww_k$ denote the set
  of cyclically reduced words of length $k$. We identify
  elements of $\Ww_k$ that differ only by rotation and inversion
  and denote the resulting
  set by $\Ww_k/D_{2k}$, where $D_{2k}$ is the dihedral group acting
  on the set $\Ww_k$ in the natural way. Let $\Ww'=\bigcup_{k=1}^{\infty}\Ww_k/D_{2k}$, and let
  $\Ww'_K=\bigcup_{k=1}^K\Ww_k/D_{2k}$.
For each $k$-cycle in $G_n$ we associate an element in $\Ww_k/D_{2k}$ formed by starting the walk
  at any point in the cycle and walk in either of two directions. Two cycles are 
  considered equivalent if they both map to the same equivalent class of words.

  \begin{defn}[Properties of words]\label{defn:propertiesw}
     For any $k$-cycle
     in $G_n$, the element of $\Ww_k/D_{2k}$ given by walking
     around the cycle is called the \emph{word} of the cycle  (see 
     Figure~\ref{fig:cycleword}).
  For any word $w$, let $\abs{w}$ denote the length of $w$.
  Let $h(w)$ be the largest number $m$ such that
  $w=u^m$ for some word $u$.  If $h(w)=1$, we call $w$ \emph{primitive}.
  For any $w\in\Ww_k$, the orbit of $w$ under the action 
  of $D_{2k}$ contains $2k/h(w)$ elements, a fact which we will frequently use.
 The sign of a letter in a word is $+1$ or $-1$ depending on whether the letter is $\pi_i$ or $\pi_i^{-1}$. 
 Let $b(w)$ denote the number of letters in $w$ whose sign is the same as the letter appearing right before it.
 In other words, $\abs{w}-b(w)$ is the number of successive sign changes in $w$.   
  Let $c(w)$ denote the number of pairs of double letters in $w$, i.e.,
  the number of integers $i$ modulo $|w|$ such that $w_i=w_{i+1}$.
  
  For example, if $w=\pi_1\pi_1\pi_2^{-1}\pi_2^{-1}\pi_1$, then $b(w)=2$ and $c(w)=3$. See Figure \ref{fig:cycleword} for another example. We will consider $|\cdot|$, $h(\cdot)$, $b(\cdot)$ and $c(\cdot)$ as
  functions on $\Ww_k/D_{2k}$, since they are invariant
  under cyclic rotation and inversion.
  \end{defn}
  
  To more easily refer to words in $\Ww_k/D_{2k}$, choose some
  representative $w_1\cdots w_k\in\Ww_k$ for
  every $w\in\Ww_k/D_{2k}$.  Based on this, we will often
  think of elements of $\Ww_k/D_{2k}$ as words instead of equivalence
  classes, and we will make statements about the $i$th letter
  of a word in $\Ww_k/D_{2k}$.  For $w=w_1\cdots w_k\in\Ww_k/D_{2k}$,
  let $w^{(i)}$ refer
  to the word in $\Ww_{k+1}/D_{2k+2}$ given by
  $w_1\cdots w_i w_i w_{i+1}\cdots w_k$.  We refer to this
  operation as \emph{doubling} the $i$th letter of $w$.
  A related operation is to 
  \emph{halve} a pair of double letters, for example producing
  $\pi_1\pi_2\pi_3\pi_4$ from $\pi_1\pi_2\pi_3\pi_4\pi_1$. 
  Since we apply these operations
  to words identified with their rotations, we do not need
  to be specific about which letter of the pair is deleted.

  
 Finally, we adopt the following convention regarding words. As our permutations get modified by multiplication with transpositions, we do not change the notation of the letters. That is, at time zero, the letters are $\left\{ \pi_i, \pi_i^{-1},\; 1\le i\le d  \right\}$ and they remain so for all time $s$ regardless of the fact that each $\pi_i$ is now modified to $\sigma_s \cdot \pi_i$. Therefore, the edge labels due to a word, say $\pi_1\pi_2^{-1}\pi_3\pi_1$, at any time $s$ should be understood as being given by the resulting permutation at time $s$.

Consider a word $w \in \Ww_k/D_{2k}$, and let $C_w(t,s)$ denote the number of cycles with word $w$ that exists in $G(t,s)$. We can now formulate a regime where the limit of the processes $C_w(\cdot,\cdot)$ can be described. Fix some $T >0$ and consider the doubly indexed process $\left(  C_w\left( (T+t)_+, s  \right),\; w\in \Ww', \; t\le 0,\; s\ge 0   \right)$ where $x_+:=\max(x,0)$. 

We will take limit as $T$ goes to infinity to get a process 
\eq\label{eq:indexlimit}
\left(  N_w(t, s),\; w\in \Ww',\; t\le 0, \; s\ge 0 \right)
\en
which is described below. We imagine that in dimension this process is indexed by the negative half-line with a \textit{front} at zero. In dimension, this process will be stationary. Recall the following Markov chain from \cite[Lemma 13]{JP14}. 

\begin{defn}[The halving chain]\label{defn:mconwords}
Consider a time homogenous Markov chain with RCLL paths on the state space $\Ww' \cup \{\Delta \}$, where $\Delta$ is the cemetery. Let $u\in \Ww_{k-1}/D_{2k-2}$ and $w\in \Ww_k/D_{2k}$ be two words such that $u$ can be obtained from $w$ by halving $j$ different pairs of letters. The transition kernel of the chain is described below.  
\begin{enumerate}[(i)]
\item The chain jumps from $w$ to $u$ at rate $j$.  
\item The chain jumps from $w$ to $\Delta$ (i.e., gets killed) at rate $\left(  \abs{w} - c(w) \right)$.
\end{enumerate}
\end{defn}
The following definition encapsulates the following simple idea. At dimension $0$, the number of cycles with word $w$ is a \bad chain running in stationarity. Once born, each cycle, looked backward in dimension performs the halving chain. If the same cycle exists at two different time points, then we observe an identical backward path. Different cycles behave independently.
\begin{defn}[The limiting process]\label{defn:limitpoisson}
Consider a Poisson point process (PPP) $\chi$ on $(-\infty, \infty) \times [0, \infty) \times \Ww'$ with an intensity measure that can be described in the following way. For any word $w\in \Ww'$, atoms appear on $(-\infty, \infty) \times [0, \infty)\times \{w\}$ with a rate given by the product of Lebesgue measure on $(-\infty, \infty)$, the exponential probability measure of rate $2b(w)$ on $[0, \infty)$, and the delta mass $2b(w)/h(w)\delta_w$. 

Consider an extension of this probability space to support independent halving chains starting from every atom of $\chi$. For an atom labeled $(z,v,w)$, we will call this chain $X_{z,v,w}(u)$, where $u\ge 0$ is the common `time' parameter for these chains. For $t\le0, s\ge 0$ and $\bar\omega\in \Ww'$, define the collection of random variables
\eq\label{eq:definecount}
N_{\bar\omega}(t,s)\stackrel{\Delta}{=} \sum_{(z,v,w)\in \Lambda^*(t,s)} 1\left\{ X_{z,v,w}(-t) = \bar\omega  \right\}. 
\en
Here $\Lambda^*(t,s)$ is the collection of all atoms $(z,v,w)$ in $\chi$ such that $z \le s \le z+v$. This is what we will refer to as the limiting process or the limiting field. 
\end{defn} 
The birth of a cycle is captured by the first coordinate of any atom of $\chi$, the second notes its lifetime, while the third is the word of the cycle. \\

We are now ready to formally state the main results in this paper.
\section{Main Results}
Our first result describes marginal cycle counts. Here and below, the topology of weak convergence is a natural generalization of the Skorokhod topology in higher dimensions as described in \cite{neuhaus71}. We will provide more details later. 

\begin{thm}\label{prop:weakconv}
Fix any $T_0, S_0 > 0$. As $T$ tends to infinity, the cycle counting field,
\[
\left( C_w(T+t,s), \quad (t,s) \in [-T_0,0]\times [0, S_0], \; w\in \Ww'\right),
\]
converges weakly in to the field $\left( N_w(\cdot, \cdot),\; w\in \Ww'    \right)$ defined in Definition \ref{defn:limitpoisson}. 
\end{thm}

Fix any $k\in \NN$. Define a family of random variables
\eq\label{eq:whatisnk}
N_k(t,s):= \sum_{w\in \Ww_k/D_{2k}} N_w(t,s), \qquad t \le 0, s \ge 0. 
\en
Thus, $N_k(t,s)$ counts the number of $k$ cycles at dimension $t$ and time $s$. 

Let
\begin{align}\label{adk}
a(d,k)=\begin{cases}
(2d-1)^k - 1 + 2d, \quad \text{when $k$ is even,}\\
(2d-1)^k + 1, \quad \text{when $k$ is odd}. 
\end{cases}
\end{align}
It is shown in \cite[Lemma 41]{DJPP} that $a(d,j)/2j$ is the size of $\Ww_j/D_{2j}$. 

For every fixed $(t,s)$, the variable $N_k(t,s)$ has the limiting law of the number of $k$-cycles in random regular graph of degree $2d$. It follows from \cite{DJPP} that every $N_k(t,s)$ is Poisson with mean $a(d,k)/2k$. Let 
\eq\label{eq:scaledfield}
X_k(t,s)= (2d-1)^{-k/2}\left(  2kN_k(t,s) - a(d,k)  \right).
\en
Then, $X_k$'s are centered random variables with variance one.  

It has been shown in \cite{JP14} that, for every fixed $s$, the vector-valued process $\left(X_k(\cdot,s), \; k\in \NN\right)$ converges in law to a family of independent stationary Ornstein-Uhlenbeck (OU) processes. Here we show surface convergence. 
Recall that a Yule process $\xi$ is a pure-birth process on $\NN$ with generator
\[
Lf(k)= k\left( f(k+1) - f(k)  \right), \quad k\in \NN. 
\]

\begin{thm}\label{prop:covariance} As $d\rightarrow \infty$, the field $\left( X_k(t,s),\; t\le 0, s\ge 0,k\in \NN  \right)$ converges weakly to a family of continuous Gaussian surfaces $\left(  U_k(t,s), \; t\le 0, s\ge 0, k\in \NN  \right)$ over any compact rectangle.
\begin{enumerate}[(i)]
\item If $j \neq k$, then $U_j$ and $U_k$ are independent. 
\item Suppose $j=k$. Consider, two elements $(t_1, s_1)$ and $(t_2,s_2)$ in $(-\infty,0]\times [0, \infty)$. If $\theta:=-\max(t_1, t_2)$ and $s:=\abs{s_1-s_2}$ then
\eq\label{eq:covriancejeqk}
\begin{split}
\Cov\left( U_j(t_1, s_1), U_j(t_2, s_2)  \right)= 2j e^{-j\abs{t_1-t_2}}  \left[  \E \left(e^{-2s\xi(\theta)}\right)   \right]^j \E\left( e^{2s\tau}  \right),
\end{split}
\en 
where $\xi(\theta)$ is the state of a Yule process at time $\theta$, starting from $\xi(0)=1$, and $\tau$ is the random variable that counts the number of sign changes along a $j$-cycle if we attach random i.i.d. $\pm 1$ at every vertex.
\end{enumerate}
\end{thm}

Our next results makes precise the idea of running time infinitesimally slowly. 
\begin{thm}\label{prop:finalgauss}
For every finite rectangle in $ \rr \times [0,\infty)$, consider $T_0$ large enough such the following process is well-defined:
\[
\left(  U_j\left(  -T_0 + u, \frac{1}{2}v e^{-T_0}    \right), \; j\in \NN  \right)
\]
where $(u,v)$ lies in the rectangle.
As $T_0$ tends to infinity, the weak limit of this field is another family of Gaussian surfaces $\left( G_j(u,v),\; u\in \rr,v\ge 0   \right)$, independent for each $j$, with the following non-trivial covariance structure: 
\eq\label{eq:finalcov}
\Cov\left(  G_j(u_1,v_1), G_j(u_2,v_2)   \right)= 2j\left( \frac{ e^{-\abs{u_1-u_2}}}{ 1 + \abs{v_1-v_2}e^{-\max(u_1, u_2)}}\right)^j.
\en
\end{thm}

Note that, for fixed time parameter $v_1=v_2$ the process is a stationary OU process. For a fixed dimension $u=u_1=u_2$, the process is a stationary Gaussian process. In conclusion, this Gaussian field can be thought of as the asymptotic fluctuation of cycle counts in the heuristic set-up described in the very beginning. 
\medskip

We now focus on eigenvalues of $G(t)$. For any $d$-regular graph on $n$ vertices $G$
and function $f\colon\rr \rightarrow \rr$, define the random variable
\[
\tr f(G):=\sum_{i=1}^n \hat{f}(\lambda_i)
\]
where $\lambda_1\ge \ldots \ge \lambda_n$ are the eigenvalues of adjacency matrix of $G$
\textit{divided} by $2(2d-1)^{1/2}$ and $\hat{f}$ is $f$ with the constant term adjusted.  The details, similar to \cite{JP14}, will be described later.  
By a polynomial basis we refer to a sequence of polynomials $\{ f_0\equiv 1, f_1, f_2, \ldots \}$ such that $f_k$ is a polynomial of degree $k$ of a single argument over reals. 

\begin{thm}\label{prop:eigenvalues}

There exists a polynomial basis $\{ f_i,\; i\in \NN\}$ 
(depending on $d$)
such that for any $K\in \NN, T_0 > 0, S_0 > 0$, the process 
\[
( \tr f_k(G(T+t, s)), \; k \in [K], \; t \in [-T_0, 0], s\in [0, S_0] )
\]
converges in law, as $T$ tends to infinity, to the limiting field $(N_k(t,s),\; k,t,s )$ of Proposition~\ref{prop:weakconv}.
Hence, for any polynomial $f$, the process $\big(\tr f(G(T+t, s))\big)$
converges to a linear combination of $(N_k(t, s),\,k,t,s)$.
\end{thm}

For our final result we will take $d$ to infinity. We will make the following notational convention: for any polynomial $f$, we will denote the corresponding linear combination from Theorem \ref{prop:eigenvalues} by $\tr f\left(G(\infty + t, s) \right)$. 

\begin{thm}\label{prop:chebycov}
Let $\{T_k, \; k \in \NN\}$ denote the Chebyshev orthogonal polynomials of the first kind on $[-1,1]$. As $d$ tends to infinity, the collection of processes
\eq\label{eq:chebyfield}
\left( \tr T_{k} \left(G(\infty + t, s) \right) - \E\left[ \tr T_{k} \left(G(\infty + t, s) \right)\right],\; t\le 0,\; s\ge 0, \; k\in \NN \right)
\en
converges weakly to the Gaussian field $U/2$, as in Theorem \ref{prop:covariance}. In particular, under the set-up of Theorem \ref{prop:finalgauss}, (half of) the same weak limit holds. 
\end{thm}

The independent Gaussian evolution of Chebyshev polynomials in \eqref{eq:finalcov} is similar to Borodin's result on the dynamics of minor processes of stochastically evolving Wigner random matrices. See \cite[Proposition 3]{Bor2}. This explains how the random transposition chain preserves the GFF fluctuations of eigenvalues in a manner similar to Dyson Brownian motion. A more detailed comparison is given below.

\subsection{Relevant existing literature and comparison with Wigner}
In this subsection we discuss analogies between this work and \cite{Bor2}. Some of the notations below is borrowed from that source. We keep our description informal. The reader is also encouraged to look at the introduction to \cite{JP14}. Consider two families of independent identically distributed real-valued stochastic processes $\left\{ Z_{ij}(t),\; j>i \ge 1,\; t\in \rr\right\}$ and $\left\{ Y_i(t),\; i\ge 1,\; t\in \rr  \right\}$ which have zero mean and suitable higher-moment conditions. Assume that there is a kernel $c(s,t)$ such that $c(s,t)\ge 0$, $c(t,t)=1$ such that 
\[
\begin{split}
\E\left(  Z_{12}(s) Z_{12}(t) \right)&\equiv c(s,t)\equiv \frac{1}{2}\E \left(  Y_1(s) Y_1(t) \right),\\
\E\left( Z^2_{12}(s) Z^2_{12}(t)  \right)&\equiv 2c^2(s,t) + 1. 
\end{split}
\]

Define a process of infinite Wigner matrices $X(t)$ by 
\[
X_{ij}(t)=\begin{cases}
Z_{ij}(t),& i < j,\\
Y_i(t),& i=j,\\
X_{ji}(t),& i > j.
\end{cases}
\] 
When the entry processes are standard Ornstein-Uhlenbeck processes on $\rr$, the spectrum evolves as a process known as Dyson Brownian motion. In this setting at a single time point the matrix arising forms the Gaussian Orthogonal Ensemble (GOE) ($\beta=1$). \\
For any $n\in \NN$, let $X(n,t)$ denote the $n\times n$ principal submatrix of $X(t)$. 
Fix a parameter $L$ that will be sent to infinity. Let $z$ be a complex number in the upper half plane $\mathbb{H}$. Let $y=\abs{z}^2$ and $x=2\Re(z)$. The height function $H_L$ of the eigenvalue distribution of $X(t)$ is a function indexed by $\mathbb{H} \times \rr$ given by
\[
H_L(z,t)=\sqrt{\frac{\pi}{2}} \#\left\{ \text{eigenvalues of $X\left( \lfloor Ly\rfloor ,t \right) \ge \sqrt{L} x$ }  \right\}.
\]
Then, Borodin shows that, as $L$ tends to infinity, $\left\{ H_L(z, t),\; z\in \mathbb{H},\; t\in \rr  \right\}$, seen as a stochastic process of random distribution on $\mathbb{H}$, converges in law to a generalized Gaussian process on $\mathbb{H} \times \rr$ whose every $t$ marginal is the GFF on $\mathbb{H}$ with zero boundary condition. 

The law of this limiting Gaussian process can be characterized by Chebyshev polynomials (\cite[Proposition 3]{Bor2}). In short, consider Chebyshev polynomials $\left(T_j,\; j \in \NN\right)$. Consider the corresponding fluctuations of linear eigenvalue statistics:
\[
\tr\left( T_j\left(X\left( \lfloor Ly \rfloor, t \right)\right)  \right) - \E\left[ \tr\left( T_j\left(X\left( \lfloor Ly \rfloor, t \right)\right)  \right)  \right].
\]   
Borodin shows, as $L\rightarrow \infty$, that the collection of limiting centered Gaussian surfaces, indexed by $\NN$, are independent of one another. For a given $j$, two points on the surface $(y_1, t_1)$ and $(y_2, t_2)$ have a non-trivial covariance given by 
\eq\label{eq:wignercov}
\frac{j}{2} \left(  \sqrt{\frac{y_1}{y_2}} c(t_1, t_2)  \right)^j, \quad y_1 \le y_2.
\en

Consider now Theorem \ref{prop:eigenvalues} and the final Gaussian field in Theorem \ref{prop:finalgauss} (divided by $2$). Not only the Gaussian surfaces for different Chebyshev polynomials are independent, but for a given polynomial the covariance structure is almost identical to \eqref{eq:wignercov}. There are two differences though. One, there is a re-parametrization of $y_i=e^{2u_i}$ which gives us
\[
\sqrt{\frac{y_1}{y_2}}= e^{-\abs{u_1-u_2}}, \quad y_1\le y_2. 
\]  
More importantly, if one takes the only possible choice of $c(v_1, v_2)=\left( 1 + \abs{v_1-v_2}\right)^{-1}$, there is an additional term that does not match with \eqref{eq:wignercov}. 

However, if $u_1=u_2=u$ and we redefine $v$ to $v e^{-u}$, it does give us the correct expression. The heuristic explanation is that while in the case of Borodin every minor process is moving in time at the same speed, for the random transpositions the speed depends on the size of the graph being considered. 

\bigskip
Our analysis is also seemingly related to a series of work by \cite{KOV} and others, 
 where a sequence of permutations satisfying condition (ii) is called
 a \emph{virtual permutation}, and the distribution on virtual permutations
 satisfying condition (i) is considered as a substitute for Haar measure on $S(\infty)$,
 the infinite symmetric group.  Although many of the same ingredients appear in both these works, our work considers several permutations while the other considers a single infinite permutations.    

Random matrix theory for sparse random regular graphs is a recent area of research which is not covered under the rubric of universality of traditional random matrix ensembles. See the discussion in the article \cite{DP}. However, empirical spectra distributions have been shown to approximate Wigner's semicircle law in different limiting regimes in \cite{DP} and \cite{TVW}. Both the above articles consider a sequence $\{d_n\}$ of degree that goes to infinity with the size $n$ of the graph, although several results in \cite{DP} extend easily to the case of a fixed degree $d$. The study of linear eigenvalue statistics for both fixed $d$ and growing $\{d_n\}$ is done in \cite{DJPP} where many similar combinatorial objects were exploited. The closest relative of the current article is \cite{JP14} which covers the dimension dynamics for fixed $d$, much as the current article. The thesis \cite[Chapter 4]{Johnson14} extends the ideas in \cite{JP14} to growing $\{d_n\}$ where a complete proof of the convergence of fluctuation of height function to the GFF has been done.

Let us also mention that when $d=1$, the process in dimension is a continuous-time version
of the CRP itself while, in time, it is the well-known random transposition Markov chain. The latter has been studied in several contexts. See, for example, the references in the book \cite{Dbook}. In particular, a long chain of literature is devoted to mixing properties of the chain. Modern bounds and more references can be found in the article \cite{BSZ11}. A slightly related study is the effect this chain has on large cycles of the permutation (the split-merge transformation). See the article \cite{DMZZ04}. In our context the case $d=1$ is unusual compared to $d> 1$. For example, $G(t,s)$ is likely to be disconnected when $d=1$ and connected when $d$ is larger. However, our results for finite $d$ continue to hold.

\section{Properties of the limiting field}

We start with the process described in Definition \ref{defn:limitpoisson}. The PPP $\chi$ and the countably many halving processes can clearly be constructed on a suitable probability space. What is not obvious is why the field $N_{\cdot}(\cdot, \cdot)$ as defined in \eqref{eq:definecount} is finite almost surely. In this subsection we prove this and other properties of the limiting field.

The following definition will be used throughout the rest of the article.

\begin{defn}\label{defn:colloquial}
A few colloquial conventions regarding an atom $(z,v,w)$ of $\chi$ or any other point process on the same space. We will refer to $w$ as the word of the atom. We say that the atom is \textit{born} during time interval $J$ if $z\in J$. We say that the atom \textit{exists} at time $s$ if $z\le s\le z+v$. The middle coordinate $v$ will be referred to as the lifetime of the atom. We will frequently use the memoryless property of the lifetime distribution without mention. Also, whenever we write `time', it refers to the time of the halving chains, which is really the dimension running backwards for the limiting field. 
\end{defn}


\begin{lemma}\label{lem:dimzero}
For any $w\in \Ww'$, the process $\left( N_w(0,s),\; s\ge 0  \right)$ is a continuous time \bad chain on the state space $\{0,1,2,\ldots\}$ and generator
\eq\label{eq:badgen0}
\frac{2 b(w)}{h(w)} \left( f(x+1)-f(x)\right)+ 2x b(w) \left( f(x-1)-f(x)\right)1_{\{x>0\}}.
\en
The chain is running in stationarity and is time-reversible. The collection of processes $\left( N_w(0,\cdot),\; w\in \Ww'  \right)$ are independent of one another. In particular, the distribution of $\left( N_w(0,s),\; w\in \Ww'   \right)$, for any fixed $s$, is the product measure of independent Poisson$\left( 1/h(w)\right)$, $w\in \Ww'$. 
\end{lemma}

\begin{proof}[Proof of Lemma \ref{lem:dimzero}] From the PPP structure, it is immediate that the processes $N_w(0,\cdot)$ are independent for various $w$. For a fixed $w$, atoms arrive at a rate $2b(w)/h(w)$ and survive an i.i.d. exponentially distributed amount of time. Clearly, $N_w(0, \cdot)$ is a continuous time Markov chain with generator \eqref{eq:badgen0}. 
This is obviously a \bad chain. Elementary arguments show that the unique stationary law is Poisson$(1/h(w))$ under which it is reversible. 
\end{proof}

The above is a particular case covered in Proposition \ref{prop:dimmarginal} stated later. However we defer stating it since one needs a few more definitions to do that.  \\
 
Now recall \eqref{eq:definecount} and the halving chain in Definition \ref{defn:mconwords}. The next lemma roughly states that given any finite (dimension, time) rectangle and words of length at most $K$, with high probability, there exists an $L$ large enough such that all such words must have shrunk from words of length at most $L$ at dimension zero.  

\begin{lemma}\label{lem:existence}  
Fix $K, L\in \NN, L > K$, let $\chi_L$ denote the restriction of $\chi$ to $[0,\infty)\times[0,\infty)\times \Ww'_L$. Define the field
\[
N^{(L)}_{\bar\omega}(t,s)\stackrel{\Delta}{=} \sum_{(z,v,w)\in \Lambda^*(t,s) \cap \chi_L} 1\left\{ X_{z,v,w}(-t)=\bar\omega   \right\}.
\]
Then, given any rectangle $R:=[-T_0,0]\times [0, S_0] \times \Ww'_K$, for $T_0, S_0 >0$ and $K\in \NN$ and any $\epsilon >0$, there exists an $L \gg K$ such that 
\eq\label{eq:LapproxN}
\P\left(  N_{w}(t,s)= N^{(L)}_w(t,s),\; (t,s,w)\in R   \right) \ge 1 - \epsilon. 
\en 
In particular, the random variable $\sup_{(t,s,w)\in R}N_{w}(t,s)$ is almost surely finite. 
\end{lemma}
\textbf{Step 2} The proof of \cite[Theorem 16]{JP14} proves the one dimensional version of the above statement. For the benefit of the reader we use similar notations since the arguments are quite similar with necessary generalizations. The basis of the argument is careful counting followed by the union bound.  
\begin{proof}[Proof of Lemma \ref{lem:existence}]  

  Fix $L > K$. Consider a word $w$ of size $l \ge L$ in an atom of $\chi$ born during $[0, S_0]$. Suppose a halving chain starting at $w$ reduces to a word $\bar\omega\in \Ww'_K$ by `time' $T_0$. Then, it has to halve at least $l-K$ times during $[0, T_0]$. To bound this probability, we recall the transition kernel of the halving chain given in Definition \ref{defn:mconwords}. An easier description of the halving probability is to attach independent exponential clocks of rate $1$ at every letter $i$ such that $w_i=w_{i+1}$, modulo the length. These are the positions that can be halved. Whenever a clock rings, we erase that letter and the state of the chain has jumped to a new word. Of course, the chain can be killed at any time, but since we are only interested in upper bounds on cycle counts, we can ignore this event. 
  
Let $E(l)$ be the event that for some word of length $l$ existing in $\chi$ during $[0,S_0]$ jumps at least $l-K$ times by `time' $T_0$. For $l \ge L$, word $w\in \Ww_l/D_{2l}$, and $I \subseteq [l]$, let $F(w,I)$ denote the event that there exists an atom during $[0, S_0]$ of word $w$ such that the  halving chain starting from it deletes all the vertices in $I$ by `time' $T_0$.

Then by union bound
\eq\label{eq:firstunion1}
\P\left( E(L)  \right) \le \sum_{w, I} \P\left[  F(w,I)  \right],
\en
where the sum is over all words of length $ l \ge L$ and all possible subsets $I \subset [l]$ such that $\abs{I} =l-K$, $w_i=w_{i+1}$ for $i\in I$. Note that such a set $I$ need not exist for all $w$ of length $l$. 
We now bound $P\left[  F(w,I)  \right]$. Let $H_w$ denote the number of atoms $(z,v,w)$ in $\chi$ that exists at any time during $[0, S_0]$. By definition of the halving chain in Definition \ref{defn:mconwords} and union bound again, we get 
\eq\label{eq:secondunion1} 
\P\left[  F(w,I)  \right]\le \left( 1 - e^{-T_0} \right)^{l-K} \E\left( H_w \right). 
\en 
since for every atom the chance that the vertices corresponding to $I$ will be deleted by `time' $T_0$ and the expected number of atoms of with word $w$ during $[0, S_0]$ is $\E\left( H_w \right)$.  
Now $\chi$ is a PPP whose intensity is given in Definition \ref{defn:limitpoisson}. Let $H_w$ consists of atoms that are born during $(0, S_0]$ and those born before $0$ but exists at a positive time. The expected number of the former is $2b(w)/h(w) S_0 \le 2 \abs{w} S_0$. For the latter we observe the following fact that $N_w(0,\cdot)$ is a \bad chain running in stationarity. 
By Lemma \ref{lem:dimzero} the distribution of the number of atoms of word $w$ that exists at $0$ is Poisson$(1/h(w))$. Hence, the expected value is at most one.  Combining the two, we get 
\eq\label{eq:somecount}
\E\left( H_w  \right) \le 2 S_0 \abs{w} + 1. 
\en  

All that remains is to bound the number of possible $w$ and $I$ over which the sum in \eqref{eq:firstunion1} runs. Now the number of $w,I$ pairs with $|w|=l$ is at most ${(2d)}^{K}l^{K}$. See \cite[page 1416]{JP14}. We include the short proof for completeness. 
For any pair $w,I$ let $u=u_1u_2\ldots u_K$ be the word of length $K$ obtained from $w$ after deleting the vertices in $I$. Hence $w$ must necessarily look like 
\begin{equation}\label{candidate}
{{\underbrace{u_1\ldots u_1}} \atop a_1 \mbox{ times}}\quad{{\underbrace{u_2\ldots u_2}}\atop{a_2 \mbox{ times}}}\,\, \ldots \,\, {{ \underbrace{u_K\ldots u_K}}\atop{a_K \mbox{ times}}}.
\end{equation}
The total number of choices for $u$ is at most ${(2d)}^K$ and the total number of choices of $a_1,a_2\ldots a_K$ is at most $l^{K}$.

Thus combining this and \eqref{eq:somecount} we get
\eq \label{sum1}
\P\left[  E(L) \right] \le  {(2d)}^{K}\sum_{l=L}^\infty \left( 2S_0 l + 1  \right) l^{K} \left(  1- e^{-T_0} \right)^{l-K}.  
\en
The right side of the above bound is summable. Hence, one can find $L$ large enough such that it is smaller than any $\epsilon >0$. 

Outside this event of probability at most $\epsilon$, no atom of word length more than $L$ contributes to $N_w(t,s)$, $(t,s,w)\in R$. This proves \eqref{eq:LapproxN}. The almost sure finiteness follows immediately since $\sup_{(t,s,w)\in R}N^{(L)}_w(t,s)$ is obviously finite for every $L$. 
\end{proof}

\begin{defn}[The doubling chain]\label{defn:doubling} The doubling chain is a time homogenous Markov process on the state space $\Ww'$ with the following transition kernel. Let $u\in \Ww_{k}/D_{2k}$ and $w\in \Ww_{k+1}/D_{2k+2}$ be two words such that $w$ can be obtained from $u$ by doubling $a$ different letters. Then, the chain jumps from $u$ to $w$ at rate $a$.
\end{defn}

An easier description of doubling chain is to attach i.i.d. exponential one random clocks to every letter of the word. Whenever a clock rings, the corresponding letter doubles. In particular, if $Y$ is a doubling chain, then $\abs{Y}$ is a Yule process.

The next lemma shows how the doubling chain can be thought of as the time-reversal of the halving chain.

\begin{lemma}\label{lem:forward}
Consider the limiting field in Definition \ref{defn:limitpoisson}. Fix any $s$. The process $\left(  N_w(t,s),\; t\in \rr, w\in \Ww'  \right)$ is a time homogeneous Markov process with respect to the natural filtration running in stationarity. Looked backwards in `time', the process $\left(  N_w(-t,s),\; t\in \rr, w\in \Ww'  \right)$ is again a time homogenous Markov process running in stationarity. Forward in `time', the process counts existing atoms of the PPP $\chi$ performing independent doubling chains while new atoms get born independently. Backward in `time', individual atoms perform independent halving chains. In particular, for any $t\in \rr$, the distribution of the vector $\left(  N_w(t,s), w\in \Ww'  \right)$ is always the product of Poisson$(1/h(w))$, $w\in \Ww'$.
\end{lemma}

\begin{proof}
This is essentially the equivalence of \cite[Lemma 12]{JP14} and \cite[Lemma 13]{JP14}. The argument has been shown for words of bounded size (the bound is called $L$ in those Lemmas). The current lemma follows by an easy extension to $L=\infty$. 
\end{proof}

 Recall the notion of (weak) duality for Markov processes from \cite[Part III]{GS84}.

\begin{lemma}\label{lem:dual} Let $\left(Q_u,\; u\ge 0 \right)$ denote the sub-Markovian transition operator for the halving chain on $\Ww'$, and let $\left(\hQ_u,\; u\ge 0 \right)$ denote the Markovian transition operator of the doubling chain. Consider the measure $\mu_h$ on $\Ww'$ such that $\mu_h(\{w\})=1/h(w)$ for all $w\in \Ww'$. Then $Q$ and $\hQ$ are dual with respect to $\mu_h$. In other words, suppose $q_u(w,\bar{w})$ and $\hq_u(\bar{w}, w)$ are the transition probabilities corresponding to $Q_u$ and $\hQ_u$, respectively. Then, for any two non-negative functions $f, g$ on $\Ww'$, we have
\eq\label{eq:duality}
\sum_{\bar{w}\in \Ww'} \frac{f(\bar{w})}{h(\bar{w})} \sum_{w\in \Ww'} g(w)\hq_u(\bar{w}, w) = 
\sum_{w\in \Ww'} \frac{g(w)}{h(w)} \sum_{\bar{w}\in \Ww'} f(\bar{w})q_u(w, \bar{w}).
\en
\end{lemma}

\begin{proof} This follows from Lemma \ref{lem:forward}. The idea is to consider our counting processes $\left(N_w(\cdot,0),\; w\in \Ww'\right)$ during any interval of dimension of length $u$. Let us call this interval $[0,u]$, \text{increasing} in dimension from $0$ to $u$. Then atoms exist as a PPP on $\Ww'$ with intensity $\mu_h$ at `time' zero. From every atom we run an independent doubling chain till `time' $u$. During interval $(0,u)$ new atoms arrive at a certain rate and we start independent doubling chains, all stopped eventually at `time' $u$. Then, the distribution of counts of various words at `time' $u$ is again a PPP with intensity $\mu_h$. By Lemma \ref{lem:forward} the paths of atoms backward in `time' is exactly the halving chain.    

Now, consider a pair of words $(\bar{w}, w)$ as an atom itself, representing an atom existing at $\bar{w}$ at `time' zero that moves to $w$ at time $u$. If there is an atom $w$ at `time' $u$ which has no pre-image at `time' $0$, we denote it by $(\Delta, w)$, where $\Delta$ is the cemetery of the halving chain. Extend the function $f$ to $\Ww'\cup \{\Delta\}$ by taking $f(\Delta)=0$. 
Define a function $F$ on $\Ww'\cup \{\Delta \} \times \Ww'$ as $F(\alpha,\gamma)=f(\alpha) g(\gamma)$. 
Then, we can count the expected value of the sum of $F$ applied to every such atomic pairs in two ways: one, forward in `time', and the other, backward in `time'. The two sides must coincide, and this proves \eqref{eq:duality}.  
\end{proof}
\begin{rmk}
Equation \eqref{eq:duality} is a generalization of \cite[Lemma 9]{JP14} which can be recovered by taking indicator functions $f=\delta_u$, $g=\delta_w$ for any pair of words $u,w$ and then taking derivative with respect to $t$ at $t=0$.
\end{rmk}

The following proposition describes the evolution of limiting cycle counts in time at any dimension. Consider the PPP $\chi$ described in Definition \ref{defn:limitpoisson}. Fix any $t \le 0$. Given any atom $(z,v,w)$ of $\chi$, consider the independent halving chain that starts from that atom. Suppose the state of that halving chain at `time' $-t$ is $\bar{w}$. Then, extend the atom $(z,v,w)$ to $(z,v,w,\bar{w})$, and consider the point process $\chi(t):=\sum \delta_{(z,v,\bar{w})}$ by dropping the original word $w$. Clearly, $\left(N_w(t,\cdot),\; w\in \Ww'\right)$ is a function of $\chi(t)$. 

\begin{prop}\label{prop:dimmarginal}
The point process $\chi(t)$ is a PPP on $(-\infty, \infty) \times [0, \infty) \times \Ww'$ with an intensity measure that has a density with respect to the product of the Lebesgue measure on $(-\infty, \infty)\times [0, \infty)$, and the counting measure on $\Ww'$. At a section $\{z\}\times [v,\infty)\times \{w\}$ on the state space, the rate is given by 
\eq\label{eq:rateblife}
{r}_t(v, \bar{w})= \frac{2}{h(\bar{w})} \E_{\bar{w}}\left(  b(Y_t) e^{-2b(Y_t) v}  \right),
\en
where $Y_t$ is a doubling chain at `time' $-t$ and $\E_{\bar{w}}$ denotes expectation with starting state being $\bar{w}$. In particular, the rate is stationary in the time-coordinate $z$. The birth rate of atoms with word $\bar{w}$ at any time is exactly
\eq\label{eq:ratebt}
r_t(\bar{w}):= r_t(0,\bar{w}) = \frac{2}{h(\bar{w})} \left(b(\bar{w}) - \abs{w} + e^{-t} \abs{w}   \right), \quad t\le 0.
\en
\end{prop}

\begin{proof}[Proof of Proposition \ref{prop:dimmarginal}] 
The fact that $\chi(t)$ is a PPP is a consequence of Poisson thinning. Pick a word $\bar{w}$. The word of every atom $(z,v,w)\in\chi$ has a certain probability of producing $(z,v,\bar{w})\in \chi(t)$ independent of every other atom. This proves independent Poisson counts over disjoint rectangles, and hence the claim.  

We will be done once we compute the intensity measure of this PPP. It is obvious from the structure of $\chi$ that the intensity measure is translation invariant in time. Hence, we can restrict ourselves to computing rates at time $0$. 
 
Let us first evaluate the birth rates of atoms with words $\bar{w}$ at dimension $t$. By time stationarity, this is function of $(t, \bar{w})$, which we will refer to as $r_t(\bar{w})$. It follows from the PPP structure that 
\[
r_t(\bar{w})= \sum_{w\in \Ww'} \frac{2b(w)}{h(w)} q_{-t}(w, \bar{w}),
\]
where $q_{-t}$, as in Lemma \ref{lem:dual} is the transition density of the halving chain. 

This allows us to express $r_t(\bar{w})$ as the right side of \eqref{eq:duality} by taking $f=\delta_{\bar{w}}$, the indicator of the word $\bar{w}$, and $g(w)=2b(w)$. Thus, from the left side of \eqref{eq:duality} we get
\[
r_t(\bar{w})= \frac{2}{h(\bar{w})} \E_{\bar{w}}\left( b\left(Y_t\right)   \right),
\]
where $Y$ is a doubling chain and $Y_t$ refers to state of the chain at `time' $-t$. 

But, every doubling increases both the size of the word and the value of $b$ by exactly one. Thus $b(Y_t)= b(\bar{w}) + \abs{Y_t} - \abs{\bar{w}}$. However, as discussed above, $\abs{Y}$ is a Yule process, and hence $\E\left( \abs{Y_t}  \right)= \abs{w} e^{-t}$. By substituting above, we get
\eq\label{eq:ycalc}
r_t(\bar{w})= \frac{2}{h(\bar{w})} \left(b(\bar{w}) - \abs{w} + e^{-t} \abs{w}   \right).
\en
\medskip

Let us now compute the joint intensity of birth and lifetimes of atoms with word $\bar{w}$ at dimension $t$. Suppose that a word $w$ at dimension $0$ gets reduced to word $\bar{w}$ at dimension $t$. The corresponding lifetime still remains exponential with rate $2b(w)$. In particular, we see that the lifetime of atoms at any dimension $t < 0$ is not exponential, unlike the case at dimension zero.   

In fact, the intensity measure of the pair birth and lifetime of an atom with word $\bar{w}$ can be easily seen from Poisson counting. Fix word $\bar{w}$ and $v> 0$. Let, as in the statement, ${r}_t(v,\bar{w})$ be the rate at which atoms with word $\bar{w}$ and lifetime in $[v, \infty)$ are getting created at dimension $-t$ and time zero. Then, as above,
\[
{r}_t(v, \bar{w})= \sum_{w\in \Ww'} \frac{2b(w)}{h(w)} e^{-2b(w)v} q_{-t}(w, \bar{w}). 
\]  
This is the right side of \eqref{eq:duality} when we take $f=\delta_{\bar{w}}$ and $g(w)=2b(w) e^{-2b(w) v}$. 

Therefore, by Lemma \ref{lem:dual}, if $Y_t$ is the state of a doubling chain at `time' $-t$, then 
\[
\begin{split}
{r}_t(v, \bar{w})&= \frac{2}{h(\bar{w})} \E_{\bar{w}}\left(  b(Y_t) e^{-2b(Y_t) v}  \right).
\end{split}
\]
The above can again be computed explicitly in terms of Yule processes, but this is unnecessary for our analysis. 
\end{proof}

\subsection{The topology of convergence} We have stochastic processes with multidimensional parameters. The topology of weak convergence that we choose to work with is described in \cite{neuhaus71}. This is a generalization of the usual Skorokhod space of RCLL paths. For the benefit of the reader we give a short informal introduction. For more details, please consult \cite{neuhaus71}.  

Let $I_1$, $I_2$ be two bounded and closed intervals in $\rr$. The space of surfaces that we will consider will be denoted by $D\left( I_1 \times I_2  \right)$. By shifting and scaling we can assume that $I_1=I_2=[0,1]$, and we will denote the corresponding space by $D$. 

To define elements in $D$, we define \textit{quadrants}. Fix any $(t,s)\in [0,1]\times [0,1]$. Then the four quadrants are the four open subsets of $[0,1] \times [0,1]$ given by removing the axes passing through $(t,s)$:  
\[
\left\{ [0,t)\times [0,s),\; [0,t) \times (s,1],\; (t,1] \times [0,s),\; (t,1]\times (s,1]     \right\}.
\]  
Some of these are empty when $(t,s)$ lies on the boundary of $[0,1]\times [0,1]$. 

We now generalize the RCLL property. For a function $f$ on $[0,1]\times[0,1]$, we say that its quadrant limits exist at $(t,s)$ if, for every non-empty quadrant $Q$ at $(t,s)$, and any sequence of points $\left\{ (t_n, s_n),\; n\in \NN  \right\}\subseteq Q$ such that $\lim_n (t_n, s_n)=(t,s)$, the quantity $\lim_n f(t_n, s_n)$ exists. 

This does not say anything about the value of the function at $(t,s)$. For every $(t,s)\in [0,1)\times [0,1)$, consider the special up-right quadrant $(t,1]\times (s,1]$. If either $s$ or $t$ is $1$, the special quadrant is given by considering the interval $[0,1)$ for that coordinate (instead of the empty set $(1,1]$). We say that the function $f$ is continuous is continuous from above at $(t,s)$ if the quadrant limit in the special quadrant is equal to the value $f(t,s)$.

We now define the space $D$ to be the space of all real valued functions on $[0,1] \times [0,1]$ which have quadrant limits and is continuous from above at every point. It follows that such functions are bounded, are RCLL in the traditional sense along every line parallel to the axes, and have at most countably many jumps.  

The Skorokhod topology on $D$ is an extension of the usual Skorokhod topology \cite[Section 2]{neuhaus71}. As usual, by defining a proper metric, the space can be turned to a complete separable metric space over which we can define weak convergence. We will provide citations as needed later. 

For the rest of the section we will need the product Skorokhod topology on $D^{\Ww'}$ or $D^\NN$. The notion of convergence in this product is pointwise convergence of every coordinate. In particular, we will use the fact that marginal tightness along every coordinate implies joint tightness. This is a consequence of Tychonoff's theorem.

\subsection{Gaussian limits and covariance computation} 

Recall the cycle counting field $\left( N_k(t,s)  \right)$ from \eqref{eq:whatisnk}. We now compute the asymptotic covariance of any pair of elements $(N_j(t_1, s_1), N_k(t_2, s_2))$, $t_1, t_2\in (-\infty,0]$, $s_1, s_2 \in [0, \infty)$, from this field, where the ordering of the pair implies $s_1 \le s_2$. This constitutes part of Theorem \ref{prop:covariance}. The next lemma computes this covariance for every $d$.

\begin{lemma}\label{lem:finitedcov} We have the following cases:
\begin{enumerate}[(i)]
\item If $j \neq k$ and $t_1=t_2$, then $\Cov\left( N_j(t_1,s_1), N_k(t_2, s_2)   \right)=0$. 
\item If $j\neq k$ and $t_1 \neq t_2$, then 
\[
0\le \Cov\left( N_j(t_1,s_1), N_k(t_2, s_2)   \right)\le c_0 (2d-1)^{j\wedge k},
\]
where the constant $c_0$ depends on $j,k, t_1, t_2, s_1, s_2$, but not on $d$. 
\item If $j=k$ and $t_1=t_2=t$, then  
\[
\Cov\left( N_j(t,s_1), N_j(t, s_2)   \right)= \frac{a(d,j)}{2j} \E\left[ \frac{1}{h(w)} e^{-2b(Y_t) s}     \right],
\]
where $\E$ represents the joint law of a word $w$ chosen uniformly at random from $\Ww_j/D_{2j}$ and $Y_t$ is the state of a doubling chain, starting from $w$, at `time' $-t$. 
\item Finally, if $j= k$ and $t_1\neq t_2$, then 
\[
\Cov\left( N_j(t_1,s_1), N_j(t_2, s_2)   \right)= \frac{a(d,j)}{2j} e^{-j\abs{t_1-t_2}}\E\left[ \frac{1}{h(w)} e^{-2b(Y_t) s}     \right],\; t=\max(t_1, t_2),
\]
where $\E$ is the probability measure described in (iii). 
\end{enumerate}
\end{lemma}

\begin{proof}[Proof of Lemma \ref{lem:finitedcov}] It follows from time stationarity, that, without loss of generality, we can replace the pair $(s_1, s_2)$ by $(0, s)$, where $s=s_2-s_1$.

Consider case (i). Let $t_1=t_2=t$. Suppose $j\neq k$.  then $N_j$ and $N_k$ counts atoms of the PPP $\chi(t)$ over disjoint collections of words. Therefore, they are independent and has zero covariance. 

Consider case (ii). We compute $\cov\left(N_j(t_1,0), N_k(t_2,s)\right)$. There are to sub-cases: either $t_1 < t_2$ or $t_1 > t_2$. Since the PPP $\chi$ is time-reversible, these two cases are symmetric. Hence, without loss of generality, we consider the case of $t_1 < t_2$. 

Consider $\chi(t_2)$. The atoms counted in $N_j(t_1,0)$ are obtained as a Poisson thinning of atoms of $\chi(t_2)$ that exist at time zero.  The atoms counted in $N_j(t_2,s)$ consists of two independent collections: those that exist at both time zero and time $s$, and those that were born after time zero but exist at time $s$.  Thus
\[
0\le \cov\left( N_j(t_1,0), N_k(t_2, s)   \right) \le \cov\left(  N_j(t_1,0), N_k(t_2,0)  \right).
\]
The last covariance, computed in \cite[Corollary 17]{JP14}, produces the bound. 

Now, consider case (iii). As before, we compute $\cov\left( N_j(t,0), N_j(t,s)  \right)$, where $s=s_2-s_1\ge 0$. Consider atoms of $\chi(t)$. The atoms counted in $N_j(t,0)$ can be classified in two groups: either existing simultaneously at both time $0$ and $s$, or not. The same holds for atoms counted in $N_j(t,s)$. By Poisson thinning 
\[
\cov\left( N_j(t,0), N_j(t,s)  \right)= \Var\left( N_j(t, [0,s])  \right),
\]
where $N_j(t,[0,s])$ is the number of atoms of $\chi(t)$ that exist simultaneously at both times $0$ and $s$. 

Define a Borel subset $\Gamma \subseteq (-\infty, \infty) \times [0, \infty) \times \Ww'$ by 
\[
\Gamma:=\left\{  (z,v,w):\; z\le 0, \; v \ge s-z,\; \abs{w}=j   \right\}.
\]
Then $N_j(t, [0,s])$ is the mass that the Poisson random measure $\chi(t)$ puts on $\Gamma$. In particular, it is a Poisson random variable whose expectation and variance are both given by the mass of the intensity measure on $\Gamma$.

The variance can now be computed using \eqref{eq:ratebt}:
\eq
\begin{split}
\Var&\left( N_j(t, [0,s])  \right)= \sum_{w\in \Ww_j/D_{2j}} \frac{2}{h(w)} \int_{-\infty}^0 \E_w \left[  b(Y_t) e^{-2b(Y_t) (s-z)}   \right]dz\\
&= \sum_{w\in \Ww_j/D_{2j}} \frac{2}{h(w)} \E_w\left[ b(Y_t) \int_{-\infty}^0 e^{-2b(Y_t) (s-z)} dz    \right], \quad \text{By Fubini-Tonelli}, \\
&= \sum_{w\in \Ww_j/D_{2j}} \frac{1}{h(w)} \E_w\left[ 2b(Y_t) \int_{s}^\infty e^{-2b(Y_t) u} du    \right] \\
&= \sum_{w\in \Ww_j/D_{2j}} \frac{1}{h(w)} \E_w\left(  e^{-2b(Y_t) s}   \right).
\end{split}
\en  
The claimed statement follows since the number of elements in $\Ww_j/D_{2j}$ is $a(d,j)/2j$. 

Finally, we consider case (iv). By symmetry, as in case (ii), we can assume $t_1 < t_2$. Again, by Poisson thinning, we can decompose both $N_j(t_1,0)$ and $N_j(t_1,s)$ as a sum of several independent Poisson random variables with exactly one common class counted in both of them. This is the count of all atoms in $\chi(t_2)$ that exist simultaneously at both times $0$ and $s$ at dimension $t_2$, and moreover, the halving chains starting from those atoms do not jump during `time' $[0, t_2-t_1]$. 
This is because other atoms counted in $N_j(t_1,0)$ either do not exist at time $s$ at dimension $t_2$, or must have descended from words of a bigger size at dimension $t_2$. Both collections are independent from atoms counted in $N_j(t_2,s)$.

Thus
\[
\cov\left( N_j(t_1,0), N_j(t_2,s)  \right)=\Var\left( B \right),
\]
where the random variable $B$ is Binomial, given $N_j(t_2,[0,s])$, with parameters $N_j(t_2,[0,s])$ and $\hat{p}$. Here $\hat{p}$ represents the probability that a halving chain starting from a word $w$, with $\abs{w}=j$, does not jump during `time' $t_2-t_1$. By definition, $\hat{p}= e^{-(t_2-t_1)j}$. Therefore, $B$ is Poisson with parameter $\hat{p} \E\left( N_j(t_2,[0,s])\right)$. Substituting the values from case (iii) computes the expression for its variance.  
\end{proof}

\begin{lemma}\label{lem:randomsign}
Fix $j\in \NN$. Let $W$ be a uniformly picked word in $\Ww_j/D_{2j}$. Then, as $d$ tends to infinity, the asymptotic law of $\abs{W}-b(W)$ is that of the number of sign changes along a $j$-cycle if we attach random i.i.d. $\pm 1$ at every vertex. The asymptotic law of $h(W)$ is the delta mass at one. In particular $(b(W), h(W))$ are asymptotically independent. 
\end{lemma}

\begin{proof}[Proof of Lemma \ref{lem:randomsign}]
One can imagine $W$ as a random pick from all possible cyclic words (up to equivalent classes) of length $j$, conditioned on being cyclically irreducible. A random pick from all possible cycles can be generated by picking i.i.d. elements from the $2d$ collection of letters $\{  \pi_i, \pi_i^{-1},\; i\in [d] \}$ at every edge of a $j$-cycle. The expected number of occurrences of successive letters $\pi_i \pi^{-1}$ or $\pi^{-1}\pi_i$ is $j/2d$. Therefore, by Markov's bound, the probability that such a cycle is \textit{not} cyclically irreducible is vanishing as $d$ tends to infinity. Now, when each letter is picked independently, their signs are are distributed as independent coin tosses. Hence the asymptotic law of $j-b(W)$. The asymptotic law of $h(W)$ follows by counting primitives. 
\end{proof}

\begin{proof}[Proof of Theorem \ref{prop:covariance}] By usual CLT for poisson variables and \eqref{eq:scaledfield} it is clear that there is finite-dimensional convergence of $\left( X_j(t,s),\; (t,s,j)\in (-\infty,0)\times(0,\infty)\times \NN \right)$ to a centered Gaussian field $\left( U_j(t,s),\; (t,s,j)\in (-\infty,0)\times(0,\infty)\times \NN \right)$. Let us first argue that the collection of Gaussian random surfaces has the stated covariance structure. 

It follows immediately from $(i),(ii)$ in Lemma \ref{lem:finitedcov} that, if $j\neq k$, then 
\[
\Cov\left( U_j(t_1, s_1), U_k(t_1, s_2) \right)= 4jk \lim_{d\rightarrow \infty} (2d-1)^{(j+k)/2} \cov\left( N_j(t_1, s_1), N_k(t_2, s_2)  \right)=0.   
\]
Extending the argument to linear combinations of $U_j$ and $U_k$ at different points in dimension and time proves that the entire fields $U_j$ and $U_k$ are independent. 

Now take $j=k$ and $t_1=t_2=t$. Assume as before, $s_1=0$ and $s_2=s\ge 0$. Recall from \eqref{adk} 
\[
\lim_{d\rightarrow \infty}(2d-1)^{-j}a(d,j)=1, \qquad j \in \NN.
\]
Then, from Lemma \ref{lem:finitedcov}, it follows that 
\[
\begin{split}
\Cov\left( U_j(t, 0), U_j(t, s) \right)=\frac{4j^2}{2j}\lim_{d\rightarrow\infty}   \E\left[ \frac{1}{h(w)} e^{-2b(Y_t) s} \right]=2j \lim_{d\rightarrow \infty} \E\left[ e^{-2b(Y_t)s}  \right].
\end{split}
\] 
The final equality is due to Lemma \ref{lem:randomsign}.

Now, fix $d$, and consider $b(Y_t)$. As in the derivation of \eqref{eq:ycalc}, we can write 
\[
b(Y_t)= b(W) + \abs{Y_t} - j, 
\]
where $W$ is a randomly chosen word of length $j$ and $\abs{Y_t}$ is a Yule process starting at $j$, independent of $W$. Thus
\[
 \E\left[ e^{-2b(Y_t)s}  \right]= e^{2sj} \E\left( e^{-2sb(W)} \right)\E\left( e^{-2s\abs{Y_t}} \right).
\]

Let $\xi(t)$ be a Yule process at time $-t$ starting with $\xi_0=1$. Then, we know that $\abs{Y_t}$ has the same law as the sum of $j$ many independent copies of $\xi(t)$. Therefore, $\E\left( e^{-2s\abs{Y_t}} \right)= \left[  \E\left( e^{-2s\xi(t)} \right)  \right]^j$.

Combining all the pieces, we get
\[
\begin{split}
\Cov\left( U_j(t, 0), U_j(t, s) \right)&= 2j \left[  \E\left( e^{-2s\xi(t)} \right)  \right]^j \lim_{d\rightarrow \infty}\E\left( e^{2s(j-b(W))} \right)\\
&=  2j  \left[  \E\left( e^{-2s\xi(t)} \right)  \right]^j \E\left( e^{2s\tau}  \right),
\end{split}
\]
where $\tau$ is the asymptotic law of $\abs{W} - b(W)$ as described in Lemma \ref{lem:randomsign}.
The general case follows along similar lines from Lemma \ref{lem:finitedcov} (iv). 
\bigskip

We now fix $k\in \NN$ and argue marginal tightness of the field $\left( X_k(t,s),\; t\le 0, s\ge 0  \right)$. We fix some rectangle $[-T_0,0]\times [0, S_0]$. The argument is similar to the case of $T_0=1$, $S_0=1$, which is what we assume for the rest of the proof. 

Consider the PPP $\chi$ from Definition \ref{defn:limitpoisson} and the independent halving chains starting from its atoms. Let $Q_w$ denote the law over the Skorokhod space $D[0,1]$ (for more details about the  Skorokhod space see \cite{Bil}) of the halving chain starting at word $w$. Then, one can think of the collection of atoms in $\chi$ and the halving chains as not separate entities but points of a PPP $\tchi$ on 
\[
(-\infty, \infty) \times [0, \infty) \times D[0, 1]
\]
with an intensity measure that is described below. Consider an atom $(z,v,x_w)$, where $X_w$ is a path of a halving chain starting at word $w$. It occurs at a rate that is the product of the rate of occurrence of $(z,v,w)$ in $\chi$ and $Q_w(dx_w)$.  

Now consider $N_k(t,s)$. One can write it as the sum of coordinatewise monotone processes in the following way. Let 
\[
\begin{split}
H_1(t,s)&= \sum_{(z,v,x_w)\in \tchi} 1\left\{ z\le s, z+v \ge 0, \abs{w}\ge k, \abs{x_w(-u)}=k, \text{for some}\; u\in [t,0]  \right\},\\
H_2(t,s)&= \sum_{(z,v,x_w)\in \tchi} 1\left\{ z\le s, z+v \ge 0, \abs{w}\ge k, \abs{x_w(-u)}=k-1, \text{for some}\; u\in [t,0]  \right\}.
\end{split}
\]
In other words, $H_1(t,s)$ is the cumulative count of all atoms of $\tchi$ that exist at some point during interval $[0,s]$ with word of size at least $k$ such that the halving chain from that word has size exactly $k$ at some dimension in $[t,0]$. Clearly $H_1(t,s)$ is increasing along $s$ and decreasing along $t$ (since $t<0$), and is distributed as Poisson for every fixed $(t,s)$. Similarly, $H_2$ counts those among $H_1$ that have jumped to a size below $k$. Thus, $H_1(t,s)-H_2(t,s)$ counts the number of atoms that are born or exists at some point in $[0,s]$ and are of size exactly $k$ at dimension $t$.

Similarly, let
\[
\begin{split}
H_3(t,s)&=\sum_{(z,v,x_w)\in \tchi} 1\left\{  0\le z+v \le s, \abs{w}\ge k, \abs{x_w(-u)}=k\; \text{for some}\; u\in [t,0]   \right\}\\
H_4(t,s)&= \sum_{(z,v,x_w)\in \tchi} 1\left\{  0\le z+v \le s, \abs{w}\ge k, \abs{x_w(-u)}=k-1\; \text{for some}\; u\in [t,0]   \right\}. 
\end{split}
\]
This is, $H_3(t,s)$ the cumulative count of all atoms that exist during $[0,s]$ and dies before time $s$, and is of size $k$ at some dimension larger than $t$. This is again increasing along $s$ and decreasing along $t$ and is marginally Poisson. A similar interpretation holds for $H_4$. Thus, $H_3(t,s) - H_4(t,s)$ counts atoms of size $k$ at dimension $t$ that have died during time $[0,s]$.  

Most importantly,
\eq\label{eq:diffmon}
\begin{split}
N_k(t,s)&=H_1(t,s) - H_2(t,s) - \left( H_3(t,s) - H_4(t,s)   \right),\\
\E\left(N_k(t,s)\right)&=\E\left(H_1(t,s)\right) - \E\left(H_2(t,s)\right) - \E\left( H_3(t,s)\right) + \E\left( H_4(t,s)   \right).
\end{split}
\en

We now claim that it is enough to show the tightness of each 
\[
S_i\stackrel{\Delta}{=}(2d-1)^{-1/2}\left(  H_i - \E H_i \right)
\]
in the $D$ topology. The reason is the Continuous Mapping Theorem. The limit of each $S_i$ will turn out to be a continuous Gaussian surface. Thus, under the product topology, the vector $(S_i, i\in[4])$ has an almost sure continuous limit. Also it is not hard to see that convergence to a continuous surface in the $D$ topology in \cite{neuhaus71} is the same as convergence in the uniform topology when restricted to continuous surfaces. This follows in the same way as for the classical Skorokhod topology. Thus, by Continuous Mapping Theorem, we can exchange the operations of limit and sums in $X_k=S_1-S_2-S_3+S_4$. This proves that $X_k$ is tight. 

\bigskip

The proof of tightness of every $S_i$ is similar. So we only explain in detail the case of $S_1$. The counting process $H_1$ itself can be decomposed in three parts. First, separately count of atoms that exist at time $0$, and those that were born after time zero. Second, among those born after time zero, count separately those which are born with words of size $k$ and those with size larger than $k$. 

That is, define
\[ 
\begin{split}
H_1^{(0)}(t)&= \sum_{(z,v,x_w)\in \tchi} 1\left\{ z\le 0, z+v \ge 0, \abs{w}\ge k, \abs{x_w(-u)}=k, \text{for some}\; u\in [t,0]  \right\},\\
H_1^{(1)}(t,s)&= \sum_{(z,v,x_w)\in \tchi} 1\left\{ 0< z \le s, z+v \ge 0, \abs{w}= k  \right\},\\
H_1^{(2)}(t,s)&= \sum_{(z,v,x_w)\in \tchi} 1\left\{ 0< z \le s, z+v \ge 0, \abs{w}> k, \abs{x_w(-u)}=k, \text{for some}\; u\in [t,0]  \right\}.
\end{split}
\]
The corresponding centered and scaled processes, $S_1^{(0)}$, $S_1^{(1)}$, and $S_1^{(2)}$, can be similarly defined. 

Notice that $H_1^{(0)}(t)$ does not depend on $s$. It can be extended to a surface by defining $H_1^{(0)}(t,s)\equiv H^{(0)}(t)$. It follows from \cite[Section 2]{neuhaus71} that the $D$ topology restricted to surfaces that are constant in the time axis is the usual Skorokhod topology for the process restricted to the dimension axis. It is not hard to see that $\left(S_1^{(0)}(t),\; -1\le t \le 0 \right)$ is a centered and scaled Poisson process and hence converges to Brownian motion in $D[0,1]$ and therefore has a continuous limit.  For details see \cite[page 1425]{JP14} and \cite{Johnson14}. Thus it has a continuous limit in $D$ for the entire surface. \\
Similarly $H_1^{(1)}(s)=H_1^{(1)}(t,s)$ does not depend on $t$ and a similar argument like above shows continuous limit for  $\left( H_1^{(1)}(s),\; 0\le s \le 1\right)$. 


Finally, for $S_1^{(2)}$ we relate this problem to the empirical process $X_n^F$ considered in \cite[Section 5]{neuhaus71}. The idea is the following, for every atom $(z,v,w)$ of $\tchi$ that is counted in $H_1^{(2)}$, consider the point $(t,z)$ on $[-1,0]\times [0,1]$, where $-t$ is the first `time' the chain $x_w$ hits a word of size $k$. The collection of points thus created is a PPP which can be described as the empirical process of i.i.d. many points in the following way. Condition on the number of points in the rectangle $[-1,0]\times [0,1]$. There are finitely many points which, by the PPP structure, are distributed independently and identically on the rectangle. A typical point $(T,Z)$ has independent coordinates: $Z$ is distributed uniformly over $[0,1]$, while $T$ has a continuous distribution of the first hitting time of size $k$ of a pure death chain, conditioned to be less than $<-t$.  
Thus, conditioned on the number of points, our counting process $S^{(2)}_1$ is basically the process $X_n^F(\cdot)$ in \cite[eqn. (4.2)]{neuhaus71} (see the remark preceding it). Since the number of points is Poisson with a mean going to infinity with $d$, a standard de-Poissonization argument extends the convergence argument in \cite{neuhaus71} to our case. This proves a continuous limit for $H^{(2)}_1$. 

Combining all three pieces and invoking a similar application of Continuous Mapping Theorem as before, we get a continuous Gaussian limit for $S_1$. Combining similar statements for $S_2, S_3$, and $S_4$, we get tightness for the surface $X_k$.

Since the limit of each $S_i$ is a continuous surface, the limit of $X_k$ must also be continuous. Since marginal tightness implies joint tightness in the product topology, this completes the proof of Theorem \ref{prop:covariance}.
\end{proof}

\begin{rmk}
Since the marginal distribution of the Yule process $\xi(t)$ is geometric with parameter $e^t$, each of the covariances in Theorem \ref{prop:covariance} can be computed explicitly. But again, this is unnecessary for our analysis. 
\end{rmk}

For the covariance in the stationary regime of Theorem \ref{prop:finalgauss}, we recall the following standard limit theorem for Yule processes. 

\begin{lemma}\label{lem:weakyule}
Let $\xi(\theta), \; \theta\ge 0$ be a Yule process such that $\xi(0)=1$. Then $e^{-\theta} \xi(\theta)$ converges in law to an exponential one random variable. Thus, for any $v\ge 0$, we have 
\eq\label{eq:expconv}
\lim_{\theta\rightarrow \infty} \E \exp\left( -v e^{-\theta} \xi(\theta) \right) = \frac{1}{1+v}. 
\en
\end{lemma} 
\begin{proof} The first claim is classical. The exact distribution at time $\theta$ is a geometric with mean $e^{\theta}$. See \cite[pg 122]{kf}. Convergence to an exponential one now follows from that. Equation \eqref{eq:expconv} follows from the stated weak convergence since $e^{-vx}$ is a continuous bounded function for $x\ge 0$.
\end{proof}

\begin{proof}[Proof of Theorem \ref{prop:finalgauss}] Let us show that the covariance converge to the stated limit. Pick $u_1, u_2, v_1, v_2\ge 0$, $j\in \NN$ and $T_0$ large enough. Define
\[
t_1=-T_0+ u_1,\; t_2=-T_0+u_2,\quad s_1=\frac{1}{2} v_1 e^{-T_0}, \; s_2= \frac{1}{2}v_2 e^{-T_0}. 
\]

Let $s=\abs{s_1-s_2}$. Then, obviously $\lim_{T_0\rightarrow \infty} \E\left(  e^{2s\tau} \right)=1$. 
Let $\theta=-\max(t_1, t_2)=T_0- u_1\vee u_2$. Then, by Lemma \ref{lem:weakyule}, 
\[
\begin{split}
\lim_{T_0 \rightarrow \infty} \E\left(  e^{-2s\xi(\theta)}  \right)&= \lim_{T_0 \rightarrow \infty} \E\left(  \exp\left( -\abs{v_1-v_2} e^{-T_0}\xi(T_0- u_1\vee u_2) \right)  \right)\\
&=\frac{1}{1+\abs{v_1-v_2}e^{-u_1\vee u_2}}.
\end{split}
\]
Combining all the pieces from Theorem \ref{prop:covariance}, gives us the correct covariance. 

We only need to argue tightness of each $U_j$ as $T_0\rightarrow \infty$. There are many ways to argue weak convergence of Gaussian surfaces. We choose to use \cite[Theorem 1]{DZ08}. The topologies allowed in \cite{DZ08} includes that of \cite{neuhaus71} (over continuous surfaces they are all uniform convergence). See the discussion at the beginning of Section 2 in \cite{DZ08} and the discussion following Theorem 3.1 in \cite{neuhaus71}. 

Take any sequence of $T_0$'s growing to infinity. In the notation of \cite[Theorem 1]{DZ08}, we have $d=2$, and take $\alpha=6$, $\beta=3$, and $a_{T_0}=1/T_0$. It follows from the covariance convergence that for any $(t_1, s_1)$ and $(t_2, s_2)$, the random variable $U_j(t_1, s_1) - U_j(t_2, s_2)$ is Gaussian with mean zero and a variance that is of the order 
\eq\label{eq:covorder}
O\left( \max\left( \abs{t_1-t_2}, \abs{s_1-s_2} \right)\right),
\en

where the $O$ can be taken not to depend on $T_0$. This shows 
\eq\label{kccondition}
\E {|U_j(t_1, s_1) - U_j(t_2, s_2)|}^{6}= O \left( {\abs{t_1-t_2}}^3+ {\abs{s_1-s_2}}^3 \right)
\en

and hence
condition (1) in \cite[Theorem 1]{DZ08}. Condition (2) follows since the Gaussian field $U_j$ is uniformly H\"older continuous as $T_0\rightarrow \infty$. This is a consequence of the Kolmogorov-\u{C}entsov Theorem and the uniform bound \eqref{kccondition}. For details, see \cite[page 53]{ksbrownian}. The uniform moment bound for the Gaussian random variable $U_j(0,0)$ follows from the fact that its variance has a limit as $T_0\rightarrow \infty$. These verify all the conditions for \cite[Theorem 1]{DZ08}; in particular, we obtain that the limiting Gaussian field $G_j$ is continuous almost surely.  This completes the proof.  
\end{proof}

\section{Weak convergence of cycle counts}

\subsection{Heuristic arguments on the limit}\label{hal}

   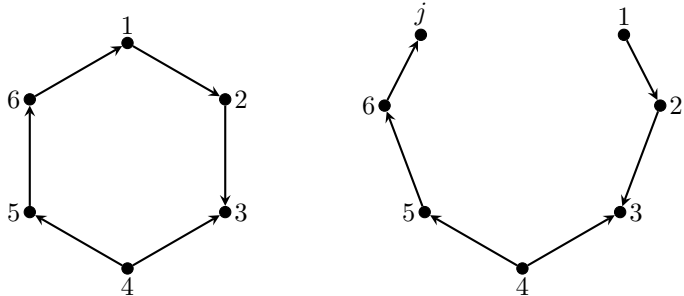
\begin{figure}[t]
      \begin{center}
        \begin{tikzpicture}[scale=1.5,vert/.style={circle,fill,inner sep=0,
              minimum size=0.15cm,draw},>=stealth]
            \node[vert] (s0) at (270:1) {};
            \node[vert] (s1) at (330:1) {};
            \node[vert] (s2) at (30:1) {};
            \node[vert] (s3) at (90:1) {};
            \node[vert] (s4) at (150:1) {};
            \node[vert] (s5) at (210:1) {};
            \draw[thick,->] (s0) to node[auto,swap] {} (s1);
            \draw[thick,<-] (s1) to node[auto,swap] {} (s2);
            \draw[thick,<-] (s2) to node[auto,swap] {} (s3);
            \draw[thick,<-] (s3) to node[auto,swap] {} (s4);
            \draw[thick,<-] (s4) to node[auto,swap] {} (s5);
            \draw[thick,<-] (s5) to node[auto,swap] {} (s0);
            \draw (s0) node[anchor=north] {$4$};
            \draw (s1) node[anchor=west] {$3$};
            \draw (s2) node[anchor=west] {$2$};
            \draw (s3) node[anchor=south] {$1$};
            \draw (s4) node[anchor=east] {$6$};
            \draw (s5) node[anchor=east] {$5$};
            
            \pgftransformxshift{3.5cm}
            \node[vert] (t0) at (270:1) {};
            \node[vert] (t1) at (330:1) {};
            \node[vert] (t2) at (20:1.3) {};
            \node[vert] (t3) at (50:1.4) {};
            \node[vert] (t4) at (160:1.3) {};
            \node[vert] (t5) at (210:1) {};
            \node[vert] (t6) at (130:1.4) {};
            \draw[thick,->] (t0) to node[auto,swap] {} (t1);
            \draw[thick,<-] (t1) to node[auto,swap] {} (t2);
            \draw[thick,<-] (t2) to node[auto,swap] {} (t3);
            \draw[thick,<-] (t6) to node[auto,swap] {} (t4);
            \draw[thick,<-] (t4) to node[auto,swap] {} (t5);
            \draw[thick,<-] (t5) to node[auto,swap] {} (t0);
            \draw (t0) node[anchor=north] {$4$};
            \draw (t1) node[anchor=west] {$3$};
            \draw (t2) node[anchor=west] {$2$};
            \draw (t3) node[anchor=south] {$1$};
            \draw (t4) node[anchor=east] {$6$};
            \draw (t5) node[anchor=east] {$5$};
            \draw (t6) node[anchor=south] {$j$};
        \end{tikzpicture}
      \end{center}
    \caption{A cycle that vanishes due to the action of the transposition $(1,j)$, $j > 6$, that gets multiplied on the left.}
    \label{fig:timecycle}
  \end{figure}

In this subsection we give heuristic arguments to justify the form of the limiting field. 
Suppose $M_T=n$, i.e., the graph $G(T,0)$ has $n$ vertices, where $n$ is very large. Consider a cycle with word $w$. For our purpose it suffices to consider that the vertex labels and edge directions are given while the edge labels are omitted. See Figure \ref{fig:timecycle} which depicts one such. For any word $w$, recall that $C_w(T,s)$ is the count of the number of cycles with word $w$ at dimension $T$ and time $s$.  

Consider the possible ways this cycle can get modified under the action of a transposition. Consider some $j > 6$. The transposition $(1,j)$ opens the cycle up (as in Figure \ref{fig:timecycle}). The same happens when $1$ is replaced by any other vertex of the cycle such that the direction of the edges is the same on both sides of it. The number of such vertices is $b(w)$. Other possibilities are multiplying with $(3,j)$ which replaces $3$ by $j$ but does not effect the count $C_w(T,\cdot)$. Nothing changes at all when multiplied with $(4,j)$. There are other possibilities. For example, we could choose a transposition $(i,j)$ where both $i,j \in [6]$. However, these events are of negligible probability for large $T$. Hence, the approximate rate at which this cycle ceases to exist is $2b(w)$ (recall definition from Definition \ref{defn:propertiesw}), at which point the count $C_w(T,\cdot)$ decreases by one. 

The other possibility is the appearance of a new cycle of word $w$. The easiest way to calculate the rate is to appeal to stationarity. From \cite[Thm 14,Cor 15]{JP14} we know that the law of $C_w(T,0)$ is approximately Poisson with mean $1/h(w)$. By exchangeability of vertex labels, the same law is true for any $C_w(T,s)$. Because the underlying graph is large and the number of cycles is roughly of constant order it is not too hard to imagine that the rate at which new cycles form should roughly stay the same along time.  Thus by the above discussion it follows that the law of $C_w(T, \cdot)$ for large $T$ is a \bad Markov chain. Since cycles disappear at rate $2b(w)$, they constant birth rate must be $2b(w)/h(w)$ in order to keep the given Poisson distribution invariant.

If cycles do not share vertices it seems reasonable that individual cycles get born and die independently of one another. Hence, the joint law of the process $\left( C_w(T,s),\; w\in \Ww',\; s\ge 0   \right)$ is approximately given by independent birth-and-death chains where the individual laws are described above.

Another way of expressing this \bad structure is to think of cycles appearing as a Poisson point process on the time axis according to a rate that depends on the word. With every atom that represents a cycle being born, we attach the length of time the cycle survives. These lifetimes are roughly independent exponentials, and this gives us the limiting PPP $\chi$ in Definition \ref{defn:limitpoisson}.

We now track these cycles backward in dimension. Suppose a cycle with word $w'$ exists at dimension $T$ and time $0$. At time $0$, looked backwards in dimension, the cycle shrinks in length or disappears entirely. The resulting sequence of words follow the halving chain. See the heuristics in \cite[Section 3.1]{JP14} and also \cite[Lemma 13]{JP14}.
Suppose now that this cycle exists simultaneously at two time points (say) $0$ and $s$. The joint law of the CRP backwards at these two time points depends on the order in which we remove the vertices. If we follow our convention outlined in \eqref{eq:orderremoval} the following convenient feature emerges. Consider the first time it gets halved, say $t_0 < T$, at which point its word becomes $w$. Then, some letter, say $\pi$, of $w$ doubles to give us $w'$ and this is the only difference between the two words. Therefore, on the cycle with word $w'$, we have a sequence of vertices $i \stackrel{\pi}{\rightarrow} j \stackrel{\pi}{\rightarrow} k$ as three successive vertices with labeled directed edges. Now, as we move time to $s$, the vertices of this cycle are exactly $\sigma_s$ applied to the vertices of $w'$. Now reduce its dimension at time $s$ and track the change in the cycle. Since $j$ is the first vertex to be deleted from the cycle at time $0$, by our convention, $\sigma_s(j)$ is the first vertex to be deleted at time $s$. Thus, the first change to $\sigma_s \cdot w'$ also happens exactly at dimension $t_0$ when we erase vertex $\sigma_s(j)$ and halve the double letters $\pi\cdot \pi$.

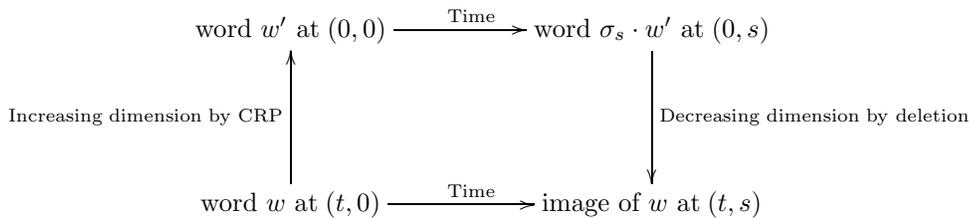
\begin{figure}[t]
\centerline{
\xymatrix@=5em{
\text{word $w'$ at}\; (0,0)  \ar[r]^{ \text{Time} } & \text{word $\sigma_s\cdot w'$ at}\; (0,s) \ar[d]^{\text{Decreasing dimension by deletion}} \\
\text{word $w$ at}\; (t,0) \ar[r]^{\text{Time}} \ar[u]^{\text{Increasing dimension by CRP}} & \text{image of $w$ at}\; (t,s)
}
}
\caption{Joint dynamics of a cycle at two times.}
\label{fig:commintw}
\end{figure}

Hence, by looking back at the most recent change at each time, inductively, allows us to describe the joint law of the process of cycles backward in the dimension parameter. Namely, at dimension $T$, consider any cycle that is born. During the entire time interval of its existence, its path, looked backwards in dimension, is identical up to relabeling of vertices (see Figure \ref{fig:commintw}). At any time point during its existence, this path is a typical path of the halving chain. This produces the limiting cycle counting field described in Theorem \ref{prop:weakconv}.  

The rest of the section is devoted to the proof of Theorem \ref{prop:weakconv}.

\subsection{Proof of Theorem \ref{prop:weakconv}} 

Notice that the time parameter for $\chi$ ranges over $(-\infty, \infty)$. This is done for a neater description. For the proof below we will work with a restricted version of $\chi$ with time varying over $[0, \infty)$. 

This is achieved by collecting all atoms of $\chi$ that \textit{exists} at time $0$ and marking them as points born at time zero. More formally, if $(z,v,w)\in \chi$ such that $z< 0$ and $z+v >0$, then, we replace this atom by another $(0, v', w)$, where $v'=z+v$. This produces an atomic intensity at time $0$. It follows from Lemma \ref{lem:dimzero} and Proposition \ref{prop:dimmarginal} that for a word $w$, the number of atoms of word $w$ at time $0$ is Poisson with mean $1/h(w)$. By the memoryless property, the remaining lifetimes of atoms remain exponentially distributed. By an abuse of notation, we will continue to call this PPP restricted $\chi$ or just $\chi$, in case there is no scope of confusion. 

\bigskip

Fix $T>0$. For every cycle appearing the process $G(T, \cdot)$, consider the triplet $(z,v,w)$ where (i) $z$ is the time when it first appears ($z=0$ if the cycle exists at time $0$), (ii) $v$ is the difference between the time it disappears and $z$ (the lifetime), and (iii) $w$ is the word of the cycle. Construct a point process $\chi(T)$ on $[0,\infty)\times[0,\infty) \times\Ww'$ as a random measure that counts these atoms $(z,v,w)$. Hence $\chi(T)$ is similar to $\chi$ but for the finite graph process at dimension $T$. For every $L\in \NN$, let $\chi_L(T)$ be the restriction of $\chi(T)$ to atoms whose words are of length at most $L$. We start by showing that every $\chi_L(T)$ converges to the claimed weak limit $\chi_L$, which is $\chi$ restricted to time $[0,\infty)$ and words of length at most $L$. 

Now, for every $T>0$, consider independent halving chains with initial condition given by atoms of $\chi_L(T)$ as described in Definition \ref{defn:limitpoisson}. There are only finitely many such chains and this operation is well-defined. One can define a cycle counting field $N^{(T,L)}_w(t,s)$ induced by these Markov chains exactly as in \eqref{eq:definecount}.

Fix any positive $T_0, S_0$. Consider both the count of cycles 
\[
C_w(T+t,s)\; \text{and}\; N^{(T,L)}_w(t,s), \quad (t,s) \in [-T_0,0]\times [0, S_0].
\]
By trivial modifications at their finitely many jump points, they can be both turned into \textit{primitive functions} in the sense of \cite[p. 1288]{neuhaus71}, and, therefore, elements in 
\[
D_2\stackrel{\Delta}{=}D^{\Ww'}\left([-T_0,0] \times [0, S_0] \right). 
\]

We have the following proposition. The topology of convergence of point processes is the usual one for random Radon measures. See \cite[Chapter 3]{extreme07}.

\begin{prop}\label{prop:mainconvg}
For any $L>0$, the point process $\chi_L(T)$ converges weakly to the restricted PPP $\chi_L$ as $T$ goes to infinity. Moreover, for any $K\in \NN$ and any $\epsilon >0$ one can find an $L \in \NN$ such that 
\eq\label{eq:LKapprox}
\limsup_{T\rightarrow \infty} \P\left(  \sup_{(t,s) \in [-T_0,0]\times [0, S_0],\; \abs{w} \le K} \abs{C_w(T+t,s) - N^{(T,L)}_w(t,s)} > 0 \right) < \epsilon.
\en 
In particular, for large enough $T$, the total variation distance between 
\eq\label{eq:LKTV}
C_w(T+t,s)\; \text{and}\; N^{(T,L)}_w(t,s), \quad (t,s) \in [-T_0,0]\times [0, S_0], \; w\in \Ww'_K,
\en
is less than $\epsilon$. 
\end{prop}


\subsection{PPP convergence at the front} We start by proving weak convergence of $\chi_K(T)$ to the limiting PPP $\chi_K$ at the front for every $K\in \NN$.

\begin{prop}\label{lem:jumpone} 
Let $\P$ refer to the law of the transposition Markov chain acting on $d$ uniform random permutations. Then, under $\P$, for any fixed $K>0$, the random measure $\chi_K(T)$, converges in law to the restricted PPP $\chi_K$.   
\end{prop}

To prove the above proposition we start with the following definitions inspired from \cite{LP} and earlier work on word maps. 

\begin{defn}\label{defn:trailsetal} Let $w$ be a word (not an equivalent class). A trail with word $w=w_1w_2\ldots w_k$ is an edge-labeled directed graph of the form   
\begin{align}\label{fig:trail}
        \xymatrix{
          s_0 \ar[r]^{w_1} & s_1 \ar[r]^{w_2}&s_2\ar[r]^{w_3}&
          \cdots\ar[r]^{w_k}&s_k
        }
\end{align}
with $s_i\in\{1,\ldots,n\}$ and does not have repeated vertices except perhaps at the two ends. The trail is said to be closed if $s_k=s_0$. Clearly, closed trails are cycles with word $w$. A pre-cycle with word $u$ is a trail that is \textit{not} closed and that can be obtained by multiplying (on left) a closed trail with word $u$ with some transposition $\sigma$. Multiplication here means the natural action of the transposition on the edges of the trail. Since multiplication by a transposition is an involution, $\sigma$ multiplied to the pre-cycle gives us a cycle with word $u$. This is only possible if the signs of $u_1$ and $u_k$ are the same, $\sigma=[s_0, s_{k}]$, where $\abs{u}=k$. Let $S_u$ denote the set of pre-cycles with word $u$. Two pre-cycles with words $u_1$ and $u_2$ are called equivalent if $u_1^{-1}=u_2$ and the sequence of vertices in the first pre-cycle is the reverse of that in the second. 
\end{defn}

Consider some $T>0$ and let $M_T$ be the number of vertices of $G(T,0)$. Condition on $M_T=n$. For the proofs in this subsection, we will send $n$ to infinity instead of $T$. This is equivalent by the well-known fact. Assume that at dimension zero, the permutations have exactly one label, $1$. Then $M_T$ is the state of a Yule process at time $T$ starting with one individual. 
We know from Lemma \ref{lem:weakyule} that the weak limit $\lim_{T\rightarrow \infty} e^{-T}M_T$ is exponential with mean one. Thus, sending $T$ to infinity is equivalent to sending $M_T$ to infinity.

We now define an appropriate filtration. Let $\mcal{G}_0$ denote the $\sigma$-algebra generated by the $d$ many permutations $\{ \pi_1^{(n)}, \ldots, \pi_d^{(n)} \}$ at time $0$. For any positive $s$, let $\mcal{G}_s$ be the $\sigma$-algebra generated by the path of the random transposition Markov chain applied to these $d$ permutations during time $[0,s]$. Let $\pi_i(s)$ denote the state of the $i$th permutation at time $s$. Then the vector-valued process $\left(  \pi_1(s), \ldots, \pi_d(s)  \right)$ is Markov with respect to this filtration.

Since $T$ will be kept implicit in the analysis in this subsection, we will shorten $C_w(T,s)$ to $C_w(s)$. Fix an arbitrary $K\in \NN$. We call a cycle (or a word) \textit{short} if its length is at most $K$.
Define the following $\left\{ \mcal{G}_s \right\}$ stopping times:
\eq\label{eq:defnst}
\begin{split}
\tau_1&= \inf \left\{ s\ge 0: \text{a newborn short cycle shares a vertex with an existing short cycle} \right\},\\
\tau_2&= \inf \left\{ s\ge 0:\; \sum_{w\in \Ww'_{2K}} C_w(s) > \sqrt{\log n}     \right\},\\
\tau_3&= \inf\left\{ s\ge 0: \; \sum_{w\in \Ww'_K }\abs{C_w(s)-C_w(s-)} > 1      \right\}.
\end{split}
\en
Note that $\tau_2$ bounds cycles of length up to $2K$ and $\tau_3$ rules out the possibility of the short cycle count jumping by more than one at any given moment in time.  

Let $\stime=\min\left\{  \tau_1, \tau_2, \tau_3 \right\}$. Assume $\stime > 0$. 

\begin{defn}\label{propgood}
Let $\theta(s)$ denote the proportion of vertices counted in short cycles at time $s$. Thus, $\theta(0)\le 2K\sqrt{\log n}/n$. 
\end{defn}

\medskip

\nin\tbf{Step 1.} Fix $s > 0$. Suppose there are $N$ short cycles in the graph at time $s$. Let $\xi_i$ be the stopping time when the $i$th short cycle vanishes.  For $h\ge 0$, let $\mcal{I}(s+h)$ be the vector of length $N$ whose $i$th coordinate, $\mcal{I}_i(s+h)$, is the indicator of the event $\{\xi_i \le s+h \}$. Consider this now as a process in $h$ that starts at the vector of all zeroes, and then, with progressing time, individual coordinates jump to one.

\begin{lemma}\label{lem:deathrate} There exist a stopping time $\tau^*$ and a family of progressively measurable nonnegative processes $\left( \lambda_w^-(\cdot),\; w\in \Ww'  \right)$ such that the following happens. 
\begin{enumerate}[(i)]
\item Until $\tau^*\wedge \stime$, the process $\mcal{I}$ is a counting process, i.e., every coordinate increases exactly by one, and no two coordinates jump together. 
\item Suppose $w$ is the word of the $i$th cycle. Then
\eq\label{eq:dyingcount}
\mcal{I}_i(s+h) - \int_0^{h\wedge \xi_i} \lambda_w^-(s+v)dv, \quad h\ge 0,
\en
is a local martingale. 
\item The death rates $\lambda_w^-(\cdot)$ satisfies the uniform estimate:
\eq\label{eq:deathest}
\lambda_w^-(s) \in 2b(w)\left[1- \frac{\sqrt{\log n}}{n}, 1\right], \qquad \text{for all $s\ge 0$}.
\en
\item Let $E^*$ denote an independent exponential random variable with mean $n/\log n$. Then, for any $t >0$, $P\left(  \tau^* \wedge \stime > t  \right) \ge P\left(  E^* \wedge \stime > t \right)$.
That is, $\tau^* \wedge \stime$ stochastically dominates $E^* \wedge \stime$. 
\end{enumerate}
\end{lemma}

\begin{proof}[Proof of Lemma \ref{lem:deathrate}]
By the homogenous Markov property, it suffices to consider the case of $s=0$. 
Thus consider all cycles existing at time $0$. Suppose $\mcal{I}(0)=x$.  Consider the probability that the next transposition will turn the count to $x+e_i$, where $e_i$ is the standard basis in $\rr^N$ and $i$ is some coordinate which is currently $0$. 
Consider the following collection of transpositions that can destroy the $i$th cycle while keeping others unchanged. Let $w$ be the word of the $i$th cycle. As explained earlier in subsection \ref{hal} and Figure \ref{fig:timecycle}, a cycle with word $w$ vanishes if one of the vertices (say $u$) incident at $b(w)$ many spots is involved in the transposition. Now given $u$, the proportion of $v$ such that the transposition $[u,v]$  leads to $x+e_i$ and $v$ is not a vertex of any short cycle is in between $1$ and $1- \theta(0)$ (recall $\theta(s)$ from Definition \ref{propgood}). Call this proportion $\delta_w(0)$.
Let $\tau^*$ denote the first time that both vertices in the transposition are selected from the short cycles. Then, until $\tau^* \wedge \stime$, the coordinates of $\mcal{I}$ do not jump together and the infinitesimal death rate at time $0$ makes sense:
\eq\label{eq:dratem}
\lambda^-_w(0):= \lim_{h\rightarrow 0} h^{-1} P\left[ \mcal{I}(h)=x + e_i, \mid \mcal{G}_0,\; \mcal{I}(0)=x   \right]= 2 b(w)\delta_w(0).
\en
Similarly, the death rate for any other time $\lambda^-_w(s):=2 b(w)\delta_w(s)$ exists and satisfies
\[
\abs{\frac{\lambda^-_w(s)}{2b(w)} - 1} \le \theta(s)=O\left(  \frac{\sqrt{\log n}}{n} \right), \quad \text{until $\tau^* \wedge \stime$}.  
\]

 Moreover, by \eqref{eq:dratem},  each
\[
\mcal{I}_i(s+h) - \int_0^{h\wedge \xi_i} \lambda_w^-(s+v) dv
\]
is a local martingale, where $w$ is the word of the $i$th cycle. 

Finally, we estimate the tails of $\tau^*$. Both vertices are selected with a probability given by the square of the total number of vertices in short cycles over $n$. Thus, by our assumption on the short cycle count, we get
\eq\label{eq:ratetaust}
\lim_{h\rightarrow 0} \frac{1}{h} \P\left(  \tau^* \le  h \mid \mcal{G}_0 \right)\le   \frac{\log n}{n}.
\en  
Let $E^*$ denote an independent exponential random variable with mean $n/\log n$. By replacing $\mcal{G}_0$ by any other $\mcal{G}_s$, it follows from \eqref{eq:ratetaust} that $\tau^* \wedge \stime$ stochastically dominates $E^*\wedge \stime$. This completes the proof of the lemma. 
\end{proof}
\medskip

\nin\tbf{Step 2.} We now consider the infinitesimal rates at which cycles are born. 

\begin{lemma}\label{lem:birthrate} 
Let $R_w(s)$ denote the number of cycles of word $w$ that have ever existed during time $[0,s]$. Then, during $[0,\stime)$, the vector-valued process $\left(  R_w(\cdot), \; w\in \Ww'_K   \right)$ is a counting process, i.e. each $R_w$ jumps exactly by one, no two coordinates jump together. Moreover there exists nonnegative progressively measurable processes $\left(\lambda_w^+(s),\; w\in \Ww'_K \right)$ such that every $w\in \Ww'_K$
\[
R_w(s) - \int_0^s \lambda_w^+(u)du
\]
is a local martingale. For $s\in[0,\stime)$ we also have the following bound. 
\[
\frac{2b(w)}{h(w)} - C_1 \left( \frac{\sqrt{\log n}}{n} \right) \le \lambda^+_w(s) \le \frac{2b(w)}{h(w)}
\]
where the positive constant $C_1$ depends only on $d$ and $K$.
\end{lemma}

That the infinitesimal rates $\lambda_w^+(u)$ exist for all $u$ is not hard to see. Consider the conditional law of the graph, conditioned on $\mcal{G}_0$. There is a certain number of transpositions which, if multiplied, increases any cycle count. Since, the vertices have independent exponential clocks, there exists an infinitesimal rate of increase of cycle counts $\lambda_w^+(u)$. We now compute it using the following lemma. Recall $S_u$ (Definition \ref{defn:trailsetal}) is the set of of pre-cycles with word $u$.

\begin{lemma}\label{lem:precyclecount}
Suppose we are given $d$ permutations on $n$ labels: $\pi_1, \ldots, \pi_d$. Consider the graph generated by these permutations. Let $A_j$ denote the set of vertices on this graph that lie on a cycle of length at most $j$. Given a short word $w$ of length $k$, there exists a positive constant $c_1$, depending on $d$ and $K$ ($1\le k \le K$), such that 
\[
2n \frac{b(w)}{h(w)} - c_1 A_k \le \sum_{u\sim w} \abs{S_u} \le 2n \frac{b(w)}{h(w)},
\]  
where $u\sim w$ means that $u$ and $w$ are in the same equivalence class in $\Ww'_k/D_{2k}$. 
\end{lemma}

\begin{proof}[Proof of Lemma \ref{lem:precyclecount}] Fix a word $w$. Consider an $u \sim w$ and some $x\in [n]$. For $\abs{S_u} > 0$, we need $u_1$ and $u_k$ to have the same sign. Therefore, the number of such $u$'s is $2b(w)/h(w)$. 

For such an $u$ we claim that there is at most one pre-cycle with word $u$ whose $j$th vertex is $x$ for $j=0,1,\ldots, k$. This is true, since the trail of the only possible pre-cycle is given by
\[
\begin{split}
s_j&=x, s_{j+1}= u_{j+1}(s_j), s_{j+2}= u_{j+2}(s_{j+1}), \ldots,\\
s_{j-1}&= u_j^{-1}(s_j), s_{j-2}= u_{j-1}^{-1}\left( s_{j-1} \right), \ldots. 
\end{split}
\]
Such a trail need not be valid due to repeated vertices (including the first and the last which will make it closed). Hence, the qualifier `at most'. However, suppose that there are two indices $i_1 < i_2$ such that $s_{i_1}=s_{i_2}$, then the closed trail from $i_1$ to $i_2$ forms a cycle of size $i_2-i_1$ which is at a graph distance at most $k$ from $x$. 

There are at most $k(2d)^k\abs{A_k}$ many vertices in the graph that are at a distance at most $k$ from a cycle of length at most $k$. For all other vertices, the trail given above is not closed, has no repeated vertices and is a pre-cycle with word $u$. 

Hence, if we define $c_0=K(2d)^K$, then, for a fixed $u\sim w$ and any $1\le k \le K$, we have
\[
(k+1)\left( n -  c_0 \abs{A_K} \right) \le  \sum_{x\in [n]} \sum_{j=0}^k \sum_{s\in S_u} 1_{\{ s_j=x \}} \le (k+1) n.
\]
On the other hand, interchanging the order of summation above we get
\[
\sum_{x\in [n]} \sum_{j=0}^k \sum_{s\in S_u} 1_{\{ s_j=x \}} = \sum_{s\in S_u}  \sum_{j=0}^k \sum_{x\in [n]} 1_{\{ s_j=x \}}= (k+1)\abs{S_u}. 
\]

Combining the two estimates we get
\[
\left( n -  c_0 \abs{A_K} \right) \le \abs{S_u} \le n.
\]
Summing up over all possible $u\sim w$ proves the lemma with $c_1=2Kc_0$. 
\end{proof}

\begin{proof}[Proof of Lemma \ref{lem:birthrate}]
We now return to computing the infinitesimal rate of increase of cycles. As before, it suffices to calculate the rate at time $0$ and appeal to the Markov property to argue for every other time. At time $0$, every possible transposition $[x,y]$, where $(x,y)$ is ordered, occurs with rate $1/n$. Fix a short word $w$. We ask: how many transpositions will turn a pre-cycle to a cycle with word $w$?  

To answer this question suppose the transposition is $[x,y]$ for $x\neq y$. Let $u \sim w$ be a word with the first and last letters having the same sign. Then, we are limited to all pre-cycles with word $u$ that either start with $x$ and end with $y$, or start with $y$ and end with $x$. Out of these, every pair of equivalent pre-cycles produce the same cycle by the same transposition. Thus, every such pair is to be counted once. Hence, the rate at which cycles with word $w$ get created is given by
\[
\begin{split}
\frac{1}{2n}\sum_{(x, y) \in [n]^2,\; x\neq y}\quad \sum_{u \sim w} \sum_{s\in S_u} \left[ 1_{\{ s_0=x, s_k=y \}} + 1_{\{ s_0=y, s_k=x \}} \right]
\end{split}
\]     

As before, we interchange the order the summation we get
\eq\label{eq:notrans}
\begin{split}
\frac{1}{2n}\sum_{u \sim w} &\sum_{s\in S_u} \sum_{x\in [n]} \sum_{y\neq x} \left[ 1_{\{ s_0=x, s_k=y \}} + 1_{\{ s_0=y, s_k=x \}} \right]\\
& =\frac{1}{2n}\sum_{u \sim w} \sum_{s\in S_u} \sum_{x\in [n]}  \left[ 1_{\{ s_0=x \}} + 1_{\{s_k=x \}} \right]=  \frac{1}{n}\sum_{u\sim w} \abs{S_u}.
\end{split}
\en

The final number has been counted in Lemma \ref{lem:precyclecount}; we get 
\eq\label{eq:rateborn}
\lim_{h\rightarrow 0+} h^{-1} \E\left[ R_w(s+h)-R_w(s) \mid \mcal{G}_s  \right]=\lambda_w^+(s)=  \frac{1}{n}\sum_{u\sim w} \abs{S_u}.
\en
Note that the bound above works for all time points, not just those before $\stime$.
The proof of the lemma now follows from Lemma \ref{lem:precyclecount} and the fact that upto $\stime$,  $\sqrt{\log n}$ is an upper bound of the number of short cycles. Also the process is a counting process until $\stime$ by definition.
\end{proof}
\begin{rmk}\label{brate1}
However note that the rate at which more than one cycle gets created is the rate at which a transposition $[x,y]$ occurs where the pair $x,y$  belong to two precycles which mean that both $x,y$ lie on a cycle of size at most $2K.$ This is the first time we need to use a bound on cycles of length $2K$ to bound rates concerning cycles of length at most $K$. This explains the definition of $\tau_2.$ Now by definition of $\tau_2$ upto $\bar \sigma$ there are at most $O(\log n)$ ( the constant in the $O$ term depends only on $K$), such transpositions. Also such a transposition can create at most $K$ many new cycles since every point is on at most $K$ precycles as discussed in the proof of Lemma \ref{lem:precyclecount}. Thus denoting for all $s\in [0,\bar \sigma)$ the rate at which exactly one cycle gets created by $\tilde{\lambda}^+_w(s)$ 
we have
$$\tilde{\lambda}^+_w(s)=\lambda_w^+(s)-O\left(\frac{\log n}{n}\right)$$ where $\lambda_w^+(s)$ appears in \eqref{eq:rateborn} and the constant in the $O$ term depends only on $K$. 
\end{rmk}
\nin\tbf{Step 3.} We define new processes by extending the birth rates $\tilde{\lambda}^+_w$ and death rates $\lambda^-_w$(Lemma \ref{lem:deathrate}) from $[0,\stime)$  to  all times in $[0, \infty)$ by defining
\[
\tilde{\lambda}_w^+(s)=\frac{2b(w)}{h(w)}, \quad \lambda_w^-(s)= 2b(w), \quad s\ge \stime,\; w\in \Ww'_K.  
\]
Consider the following time changes that are measurable with respect to the predictable $\sigma$-algebra:
\[
\Gamma^+_w(s)=\int_0^s \tilde{\lambda}_w^+(v)dv, \quad w\in \Ww'_K, \quad \Gamma_w^-(s)=\int_0^s \lambda_w^-(v)dv.  
\]

We now extend the counting process $R$ from $[0,\stime)$ to the entire time axis. Let us for the moment call it $\tilde{R}$. During time $[0, \stime)$ the coordinates of $\tilde{R}$ are exactly the same as $R$ as in Lemma \ref{lem:birthrate}. At $\stime$ in the actual process, $R_{w}$ can jump by more than $1$, however $\tilde{R}_w$ in that case does not jump at all. Lastly during $[\stime, \infty)$, each $\tilde{R}_w$ continues as an independent Poisson process with rate $2b(w)/h(w)$. Clearly, this extended process is a counting process on $[0, \infty)$. 
Thus for all $w \in \Ww'_K$ and $s\in [0, \stime]$ $$\tilde{R}_w(s)=\sum_{t\le s} \Delta R_{w}(t)\mathbf{1}(\Delta R_{w}(t)=1)$$
where for any $t$ $$\Delta R_{w}(t)=R_{w}(t)-R_{w}(t_{-})$$ is the jump size at $t.$
The above sum makes sense since the process $R_w$ has only finitely many jump points in any finite time interval.
Let $A^+_w(\cdot)$ be the compensator for the coordinate process $\tilde{R}_w(\cdot)$. Then by construction $A^+_w(s)=\Gamma^+_w(s)$ for $s\in [0, \stime)$. To see this notice that by definition the rate of increase of the process $\tilde{R}$ at any time is the rate at which one cycle is produced in the process $R$. This is $\tilde{\lambda}_w^+(s)$ for all times $s< \stime $
and equal to $2b(w)/h(w)$ on the interval $[\stime,\infty).$

\begin{lemma}\label{lem:formalpp}
It is possible to extend our current probability space to define a Poisson point processes $Q$ on $\Ww'_K \times [0,\infty)$ such that the following holds. The intensity measure for $Q$ on $\Ww_K'\times (0,\infty)$ is the $\Ww'_K$-fold product of $\leb$. At $\Ww'_K \times \{0\}$ the count coincides with the number of short cycles at time $0$. 
Moreover, if we consider the extended counting process $\left( \tilde{R}_w(\cdot),w\in \Ww'_K  \right)$  in Lemma \ref{lem:birthrate} as a point process on $\Ww'_K \times [0, \infty)$ as described above, then
\[
\tilde{R}_w(s)=Q_w\left(  A^+_w(s)  \right),\; s\in [0, \infty), \; w\in \Ww'_K. 
\]
\end{lemma}

\begin{proof}[Proof of Lemma \ref{lem:formalpp}] This is a consequence of the Watanbe-Meyer theorem \cite{Watanabe64, Meyer71} that every counting process is a time-changed Poisson process. For a proof see \cite{Brown88}. The compensator has already been defined above. The extension of the space is necessary to extend the point process on $[\stime,\infty)$.
\end{proof}

\medskip

\nin\tbf{Step 4.} We can now extend the point process $\left( R_w(\cdot) \right)$ to an extension of the point process $\chi_K(T)$ in Proposition \ref{lem:jumpone} by noting the lifetime of each cycle. Formally, for every atom $(w,s)$ in $R$, extend it to $(s,v,w)$, where $v$ is the length of time that the short cycle which gets created at time $s$ exists. If a cycle with word $w$ exists beyond $\stime$ or gets born after $\stime$, the (possibly excess) lifetime will be i.i.d. $\mathrm{Exp}(2b(w))$. We will refer to this extended process by $\overline{R}(s,v,w)$. Thus $\chi_K\equiv \overline{R}$ during $s\in [0, \stime)$.

\begin{lemma}\label{lem:pppconv}
As $T$ goes to infinity, $\overline{R}$ converges to a PPP on $[0, \infty) \times [0,\infty) \times \Ww'_K$ that has independent $\mathrm{Poi}\left(  1/h(w) \right)$ many atoms for $w$ at time $s=0$, births at rate $2b(w)/h(w)\times \leb$ and a lifetime of $\mathrm{Exp}(2b(w))$.
\end{lemma}

\begin{proof}[Proof of Lemma \ref{lem:pppconv}]
The proof follows from the coupling in Lemma \ref{lem:formalpp}, the estimate in Lemma \ref{lem:deathrate}, and Lemma \ref{lem:stimeinf} below. Since the arguments are standard we outline the major steps and skip the details. 

It suffices to argue that if we take finitely many disjoint intervals on the time line, then there are independent Poisson many births of the correct rate, and that, each such newborn cycle survives an independent exponential amount of time. As $T\rightarrow \infty$, $M_{T}=n$ goes to infinity in probability. Therefore, by Lemma \ref{lem:stimeinf} below, uniformly over compact sets in $s$ ,
\[
\lim_{T\rightarrow \infty} \Gamma_w^+(s)\stackrel{p}{=} \frac{2b(w) s}{h(w)}, \quad \forall \; w\in \Ww'.
\]
The coupling in Lemma \ref{lem:formalpp} then gives us weak convergence on the birth counts. Conditioned on the birth counts, Lemma \ref{lem:deathrate} allows us to couple each lifetime with independent exponentials. The lemma now follows from the explicit error bounds given in \eqref{eq:deathest}. 
\end{proof}

\nin\tbf{Step 5.} So far we have coupled $\chi_K$ with another point process $\overline{R}$ whose weak limit is the limiting field. The two fields are identical during time interval $[0, \stime)$. Hence, Proposition \ref{lem:jumpone} follows once we argue the following.

\begin{lemma}\label{lem:stimeinf}
\eq\label{eq:asyminf}
\lim_{n\rightarrow \infty} P\left( \stime \le S  \right)=0, \quad \text{for all $S>0$}.
\en
\end{lemma}

\begin{proof}[Proof of Lemma \ref{lem:stimeinf}] Fix $\epsilon > 0$. Consider the stopping times defined in \eqref{eq:defnst}. First consider the event $\{\stime >0\}$. The event only depends on the graph $G(n,2d)$. We know from \cite[Corollary 25]{JP14} that, as $n$ goes to infinity, asymptotically almost surely no two cycles in the graph share a vertex. We also know that the asymptotic law of the vector $\left( C_w(0),\; w\in \Ww'_{2K}   \right)$  is the product of independent $\mathrm{Poi}\left( 1/h(w) \right)$ with a convergence in total variation. In fact, we know from \cite[Corollary 24]{JP14} that
\[
P\left(\sum_{w\in \Ww'_{2K}} C_w(0) > \frac{1}{2}\sqrt{\log n} \right) - P\left( \gamma_{2K} >  \frac{1}{2}\sqrt{\log n}     \right)\le c\frac{(2d-1)^{2K}}{n} ,  
\]
where $c$ is an absolute constant and $\gamma_{2K}$ is a Poisson random variable with mean
\[
\sum_{w\in \Ww'_{2K}} \frac{1}{h(w)}= \sum_{j=1}^{2K} \frac{a(d,j)}{2j} \le \sum_{j=1}^{2K} a(d,j) \le C_0 (2d-1)^{2K},
\]
where $C_0$ is a constant that depends only on $d$. Therefore, by taking a large enough $n$ we can guarantee that the event 
\eq\label{eq:initialconfig}
A_0:=\left\{ \sum_{w\in \Ww'_{2K}} C_w(0) < \frac{1}{2}\sqrt{\log n} \right\} \cap \left\{ \tau_1 > 0\right\} 
\en
occurs with probability at least $1-\epsilon/2$. Since $\tau_3>0$, the event $\{ \stime > 0\}$ has probability at least $1-\epsilon/2$.

Start from an initial configuration that satisfies $A_0$. The stopping time $\tau_3$ guarantees that the process of short cycle counts do not jump by more than one. Consider the possible cases when $\tau_3$ happens by defining three other stoping times. 
\begin{enumerate}[(i)]
\item $\tau_4$ is the first time two short cycles with words $w$, $w'$ (possibly same) appear simultaneously.
\item $\tau_5$ is the first time two short cycles with those words disappear simultaneously. 
\item $\tau_6$ is the first time when one short cycle appears while another disappears. 
\end{enumerate}

Consider $\tilde\tau=\tau_1 \wedge \tau_2 \wedge \tau_4$. We first evaluate the probability that $\{\tilde\tau > S\}$ for any $S > 0$. Call a transposition at time $s$ \textit{bad} if multiplication by that transposition will lead to $\tilde\tau$. Hence any time $s\in [0, \tilde\tau)$ it suffices to count the number of bad transpositions. This is a computation similar to the proof of Lemma \ref{lem:birthrate} and \ref{lem:precyclecount}.

Let $A_1$ be the event that there are less than $\sqrt{\log n}/2$ many new births of cycles of size at most $2K$ during $[0,S]$. Thus, under $A_0 \cap A_1$, the cycle counts of those cycles never exceed $\sqrt{\log n}$. We estimate the probability of $A_1$. At any given $s$, we bound the total number of pre-cycles that can possibly give us one or more cycles (simultaneously) of size at most $2K$.

Given any word, it follows from \eqref{eq:rateborn} and Lemma \ref{lem:precyclecount} that the expected total number of cycles of word $w$ created during time $[0,S]$ is at most $2Sb(w)/h(w)$. For a word of length at most $2K$, we can bound this mean by $2KS$. There are at most $(2d)^{2K+1}$ many words of length at most $2K$. Thus the expected number of cycles with such words born during $[0,S]$ is bounded above by $2KS(2d)^{2K+1}$. Hence, by Markov's inequality, one can choose $n$ large enough such that the probability of more than $\sqrt{\log n}/2$ births is at most $\epsilon/4$. By a union bound, 
\eq\label{eq:a0a1bnd}
P(A_0 \cap A_1) \ge 1 - 3\epsilon/4.  
\en
\medskip

Assume that $A_0\cap A_1$ holds. Thus $\tau_2 > S$. Consider $\tau_4$. Denote the state of the permutations at $(\tau_4-, \tau_4)$ by $\left( \pi_i-,\; i\in [d]  \right)$ and $\left( \pi_i, i\in [d]\right)$. Here $\tau_4-$ refers to left limit of the chain at $\tau_4$. 
Now observe the change in reversed time. Let $\sigma$ be the transposition that occurs at $\tau_4$. By taking inverses, each $\pi_i-$ can be obtained from $\pi_i$ by left multiplication by the same transposition $\sigma$. New cycles are formed by multiplying pre-cycles with $\sigma$. Therefore we see that two cycles can appear simultaneously at $\tau_4$ if and only if at $\tau_4-$ the vertices formed a cycle of size $\abs{w} + \abs{w'}$ and the transposition chooses two elements from its vertices. Since this cycle can be of size at most $2K$ and we have assumed $\tau_2>S$, the rate at which such a transposition occurs is bounded by
\eq\label{eq:tau4rate}
\frac{1}{n}\left(\sum_{w\in \Ww'_{2K}} C_w\right)^2 \le O\left(\frac{\log n}{n}\right).
\en 
where the constant in the order term only depends on $K$.

Now consider the rate which new short cycles share vertices with existing short cycles. This happens if there is a short pre-cycle that shares a vertex with an existing cycle. The number of existing short cycles is bounded by $\sqrt{\log n}$. Each has at most $K$ vertices. It has been shown in Lemma \ref{lem:precyclecount} that each vertex can lead at most $K^2$ pre-cycles with a given short word. Hence, the rate is bounded above by $c_1 {\log n}/{n}$, where $c_1$ is a constant depending on $d$ and $K$. 

Hence, during $[0,S]$, the random variable $\tau_1 \wedge \tau_4$ stochastically dominates an exponential random variable of rate $c_2\log n/n$, where $c_2$ is some constant depending only on $d$ and $K$. Thus one can take $n$ large enough to guarantee that 
\[
\P\left( \tilde\tau > S  \right) \ge 1 - \frac{7\epsilon}{8}. 
\]

Now consider $\tau_5$. Until $\tilde\tau$ no two short cycles share a vertex. Hence $\tau_5$ happens only if the transposition selects both vertices counted in short cycles and the cycles merge. The probability of such a transposition is bounded by $\log n/n$. Thus, comparing with an exponential with rate $\log n/n$, we see that one can take $n$ large enough to guarantee that 
\[
\P\left(  \tilde\tau \wedge \tau_5 > S  \right) \ge 1 - \frac{9\epsilon}{10}. 
\]  

The case of $\tau_6$ is similar. The transposition has to involve a vertex of a small cycle and another of a small pre-cycle. In any case, the rate of such transpositions is again of the order of $\log n/n$. Thus, one can take $n$ large enough to guarantee 
\[
\P(\stime > S)=\P\left(  \tilde\tau \wedge \tau_5\wedge \tau_6  > S  \right) \ge 1 - \epsilon.
\]
Since $\epsilon$ is arbitrary, this proves our claim. 
\end{proof}

The next subsection proves Proposition \ref{prop:mainconvg} using Proposition \ref{lem:jumpone} by proving that the projection of the backward chinese restaurant process on the space of cycles converge to the halving chain defined in Definition \ref{defn:mconwords}. A one dimensional version of this result is proved in the proof of Theorem 16 in \cite{JP14} and some of the basic arguments appearing there are used in the following proof.  
\subsection{Convergence to halving chains}\label{sec:halvingconv} 

Fix time $s>0$. Consider the collection of cycles present at any time during $[0,s]$ at dimension $T$. We consider their evolution backward in dimension. We claim that these evolve as independent halving chains. Notice that we do not mention their lifetime. This is because, by construction, if the same cycle exists at two time points, their backward evolution is identical. 
Let us outline the argument  Consider the graph valued process backward in dimension 
\[
\bG{v}{u}=G(T-u,v), \quad v\in [0,s], \; u \ge 0. 
\] 
Fix $K\in \NN$. Choose a large positive integer $L\gg K$ and ignore all of $\bG{v}{0}=G(T,v)$, $v\in [0,s]$, except for the subgraph consisting of cycles of size $L$ and smaller. Call this graph $\bGraph{v}{0}$. Consider now the evolution, backward in dimension, of the graphs $\bGraph{v}{\cdot}$, $0\le v\le s$. 

Define an event $\gevent$ which states that no two of the finitely many cycles of size at most $L$ that are created in dimension $T$ and during time $[0,s]$ share a vertex. Then, our proof goes by shown (i) $\gevent$ holds asymptotically almost surely, and (ii) under $\gevent$ the analysis of the backward processes $\bGraph{\cdot}{\cdot}$ is trivial.  This suffices then by a Lemma \ref{lem:existence} type argument which says large cycles do not quickly shrink to small cycles in the halving chain.

Claim (i) is almost Lemma \ref{lem:stimeinf}. The only difference is that it is possible for two cycles existing at disjoint intervals of time to have a common vertex. But, an easy extension to the same argument covers this case. 

Let us now explain (ii).  We ignore the vertex labels and consider every vertex in the graph to have an exponential one clocks attached to it. Backward in dimension, whenever the clock of a vertex rings, we remove that vertex from every permutation. By our construction, the same vertex (or, more precisely, its image) is removed simultaneously from every point in time. The remarkable fact is that, under $\gevent$, each cycle evolves independently as a halving chain for every $T$.  

Let $R$ be the rectangle $[-T_0, 0] \times [0, S_0] \times \Ww'_K$. Now,
\[
\begin{split}
P&\left(  C_w(T+t,s) = N^{(T,L)}_w(t,s),\; \forall\; (t,s,w)\in R \right)\\
\ge& P\left(  C_w(T+t,s) = N^{(T,L)}_w(t,s),\; \forall\; (t,s,w)\in R \mid \gevent\right) P(\gevent).
\end{split}
\]
Hence we will be done once we show that, for any $\delta >0$ there exists large enough $L$ such that, 
\eq\label{eq:atleastdel}
\limsup_{T\rightarrow \infty} P\left(  C_w(T+t,s) = N^{(T,L)}_w(t,s),\; \forall\; (t,s,w)\in R \mid \gevent\right) > 1-\delta.
\en
The argument is similar to the proof of Lemma \ref{lem:existence} and hence we use similar notation. Let $E_T(L)$ be the event that some cycle of length $l>L$ that exists at dimension $T$ and anywhere in time $[0,S_0]$ at least $l-K$ of its shrinking vertices are deleted by dimension $T-T_0$.

For $l > L$, word $w\in \Ww_l/D_{2l}$, and $I \subseteq [l]$, such that $\abs{I} =l-K$, $w_i=w_{i+1}$ for $i\in I$, let $F(w,I)$ denote the event that a cycle of word $w$ that exists at some point in $\{T\}\times[0,S_0]$ shrinks all the vertices in $I$ by dimension $T-T_0$. 
Thus exactly as \eqref{eq:firstunion1}
\eq\label{eq:firstunion}
P\left( E_T(L)  \right) \le \sum_{w, I} P\left[  F(w,I)  \right].
\en
And then similarly as in \eqref{eq:secondunion1} we have
\eq\label{eq:secondunion}
P\left[  F(w,I)  \right]\le \left( 1 - e^{-T_0} \right)^{l-K} E\left( G_w \right). 
\en 
where  $G_w$ denotes the number of cycles of word $w$ that ever exists during $\{T\}\times[0,S_0]$. 

The following lemma is needed to finish the current argument. 

\begin{lemma}\label{lem:howmnycyc}
For any $w$ and any $T>0$, we must have $E(G_w) \le 2 S_0 \abs{w} + 1$. 
\end{lemma}

Keeping the proof of this lemma for last, let us continue with the previous argument. Using Lemma \ref{lem:howmnycyc} as in \eqref{sum1} we have
\eq\label{eq:thirdunion}
P\left[  E_T(L) \right] \le  
{(2d)}^{K}\sum_{l=L+1}^\infty \left( 2S_0 l + 1  \right) l^{K} \left(  1- e^{-T_0} \right)^{l-K}.  
\en
The right side of the above bound is summable. Hence, one can find $L$ large enough such that it is smaller than any $\delta >0$ proving \eqref{eq:atleastdel}. 

\begin{proof}[Proof of Lemma \ref{lem:howmnycyc}] Given a word $w$ the number of cycles to ever exist during $\{T\} \times [0, S_0]$ consists of two kinds: (i) those that exist at $(T,0)$ and (ii) those that are created during $\{T\} \times (0,S_0]$. The expected value of the first kind is at most one as argued in the second displayed equation on \cite[p. 19]{JP14}. So, we focus on the second kind. 

Consider the counting in \eqref{eq:notrans} for a word $w$ that need not be short. In any case, \eqref{eq:rateborn} continues to hold: 
\[
\lim_{h\rightarrow 0} h^{-1}\E\left[ R_w(s+h) - R_w(s)\mid \mcal{G}_s, \; R_w(s)=x  \right] \le \frac{2b(w)}{h(w)}. 
\]
Therefore, for any $w$, during time $(0,S_0]$ the expected total number of $w$ cycles created during that time is dominated by 
\[
\frac{2b(w)}{h(w)} S_0 \le 2\abs{w}S_0.
\]
This completes the argument. 
\end{proof}

\subsection{Weak convergence of cycle counts} Finally we prove Theorem \ref{prop:weakconv}.

\begin{proof}[Proof of Theorem \ref{prop:weakconv}]
Consider the set-up of Proposition \ref{prop:mainconvg}, in particular, the rectangle $R=[-T_0, 0] \times [0, S_0] \times \Ww'_K$. We have four different random fields of interest.
\begin{enumerate}[(i)]
\item The field of cycles: $\left( C_w(T+t, s),\; (t,s,w)\in R    \right)$.
\item The field $\left(  N_w^{(T,L)}(t,s),\; (t,s,w)\in R  \right)$ with non-limiting initial conditions $\chi_L(T)$.
\item We now define the field $\left( N_w^{L}(t,s),\; (t,s,w)\in R   \right)$ of limiting cycle counts with the limiting initial condition $\chi_L$. 
\item And the actual limiting field $\left( N_w(t,s),\; (t,s,w)\in R   \right)$ defined in Definition \ref{defn:limitpoisson}. 
\end{enumerate}

Proposition \ref{prop:mainconvg} and Slutsky's theorem implies that any weak limit of fields (i) and (ii) must be the same. We also know from Proposition \ref{prop:dimmarginal} that the field (iii) converges almost surely to (iv) as $L$ tends to infinity. Hence, it suffices to show that, as $T$ tends to infinity, the field (ii) converges weakly to (iii). 

We know that $\chi_L(T)$ converges weakly to $\chi_L$ as random Radon measures on the space $[0, \infty) \times (0, \infty) \times \Ww'_L$. By Skorokhod's theorem, one can construct a probability space and copies of $\chi_L(T)$ and $\chi_L$ such that this convergence holds almost surely. We will restrict ourselves to this constructed space and construct couplings. 

On this space one can find $T'$ large enough such that, for all $T \ge T'$, the total mass of $\chi_L(T)$ on every $w\in \Ww'_L$ is the same as that of $\chi_L$. This is because, almost surely, there are only finitely many of them. Assume, henceforth, $T\ge T'$.

Informally, the same number of cycles of each word get created under $\chi_L(T)$ and $\chi_L$. For every word $w$ enumerate these atoms according to the times of their births. Then, take an identical set of independent halving chains starting from $w$ and attach to the two sets of atoms in the same order. Thus, we have a collection of independent Markov chains whose birth times and lifetimes are \textit{slightly off}, however, the correct alignment renders their paths identical. This produces a coupling of field (ii) and (iii).

However, as $T$ goes to infinity, the finite vector of birth times and lifetimes of every cycle in $\chi_L(T)$ is a continuous function of the point process, converges almost surely to that of $\chi_L$. By construction of our coupling this causes the fields (ii) to converge to (iii) pointwise in the rectangle $R$. This implies converges in the space $D_2$ trivially (see \cite[eqn. (2.3), (2.4)]{neuhaus71}). 

The claimed weak convergence now follows. 
\end{proof}

\section{Linear eigenvalue statistics}
 Let us recall some of the basic facts established in \cite[Section 3, 5]{DJPP} and \cite[Section 5]{JP14} that connect linear eigenvalue statistics with cycle counts. Recall $G_n$ is a random regular graph of degree $2d$ on $n$ vertices. A closed non-backtracking walk is a walk
  that begins and ends at the same vertex, and that never follows
  an edge and immediately follows that same edge backwards.
   If the last step of a closed non-backtracking walk is anything other
  than the reverse of the first step, we say that the walk is \emph{cyclically
  non-backtracking} (CNBW). 
  Cyclically non-backtracking walks on $G_n$
  are exactly the closed non-backtracking walks whose words
  are cyclically reduced.
  Let $\CNBW[n]{k}$ denote the number of closed cyclically
  non-backtracking walks of length $k$ on $G_n$.

 Let $\{T_n(x)\}_{n \in \mathbb{N}}$ be the Chebyshev polynomials of the first kind 
on the interval $[-1,1]$.
We define a set of polynomials
\begin{align*}
  \Gamma_0(x) & = 1~, \\
  \Gamma_{2k}(x) &= 2T_{2k}(x)+ \frac{2d-2}{(2d-1)^k}~,~~\forall ~k \geq 1~, \\
  \Gamma_{2k+1}(x) &= 2T_{2k+1}(x)~, ~~\forall ~k \geq 0~.
\end{align*}

  Let $A_n$ be the adjacency matrix of $G_n$, and
  let $\lambda_1\ge\cdots\ge\lambda_n$ be the eigenvalues
  of $(2d-1)^{-1/2}A_n/2$.  
  Then
  \begin{align}
    \sum_{i=1}^n\Gamma_k(\lambda_i)&=(2d-1)^{-k/2}\CNBW[n]{k}.
    \label{eq:eigencnbw}
  \end{align}

Now, for any cycle in $G_n$ of length $j|k$, we obtain $2j$ non-backtracking walks
  of length $k$ by choosing a starting point and direction and then
  walking around the cycle repeatedly.
  It follows from \cite[Corollary~18]{DJPP}, that for fixed $d$ and $r$, all cyclically non-backtracking walks of length $r$ or less
  have this form with high probability.
  Thus the random vectors $\big(\CNBW[n]{k},\,1\leq k\leq r\big)$
  and $\big(\sum_{j\divides k}2j\Cy[n]{j},\,1\leq k\leq r\big)$ have the same
  limiting distribution, and the problem of finding the limiting distributions
  of polynomial linear eigenvalue statistics is reduced to finding limiting
  distributions of cycle counts. \\
  
  Most of this section is devoted to showing that a similar statement holds for 
  the entire two parameter field $G(\cdot, \cdot)$. Let $\mathrm{CNBW}_{k}(\cdot,\cdot)$ denote the corresponding field of cyclically non-backtracking walks of length $k$.
As in \cite{JP14} call a CNBW \textit{bad} if it is anything other than a repeated walk around a cycle. Formally we prove the following proposition.

\begin{prop}\label{prop:nobadwalk}
Fix any $T_0, S_0, \epsilon >0$ and $K\in \NN$. Consider $T$ for which the rectangle $R_0:=[T-T_0, T] \times [0, S_0]$ is a subset of $[0, \infty)^2$. Then, for all large enough $T$, the probability that there is any bad $\mathrm{CNBW}$ of length at most $K$ in the field of graphs $G(t,s)$, $(t,s)\in R_0$, is at most $\epsilon$. 
\end{prop}

To prove the above proposition we recall the definition of categories of trails from \cite[Section 3]{DJPP}. Recall the definition of trails from Definition \ref{defn:trailsetal}. In this section we will use an expanded definition that allows vertices to repeat. If we start at a vertex of a CNBW and walk around it, we get a closed trail with possible repeated vertices. Given such a trail its category is a directed, edge-labeled graph depicting the overlap of trails. Essentially one gets a category from a trail by gluing the vertices with the same label.   
See Figure 1 in \cite{DJPP} or Figure 7 in \cite{LP} for the general definition of categories of a list of trails. 
%

\begin{figure}[t]
\centerline{
        \xymatrix{
          1 \ar[r]^{\pi_1} & 2 \ar[r]^{\pi^{-1}_2}& 3 \ar[r]^{\pi_3}& 2 \ar[r]^{\pi_2} & 4 \ar[r]^{\pi_1^{-1}} & 5 \ar[r]^{\pi_3} & 1
        }
}
\vspace{1cm}

\centerline{        
        \xymatrix@=6em{
       \{s_2\} \ar @/_1.5pc/ [r]_{\pi_2} \ar @/^1.5pc/ [r]^{\pi_3} & \{ s_1, s_3\} \ar[r]^{\pi_2} & \{s_4\} \\
       & \{s_0\} \ar[u]^{\pi_1} & \{s_5\} \ar[l]^{\pi_3} \ar[u]_{\pi_1}
         }
}
\caption{A bad CNBW trail and its category graph}
\label{fig:cnbwcateg}
\end{figure}
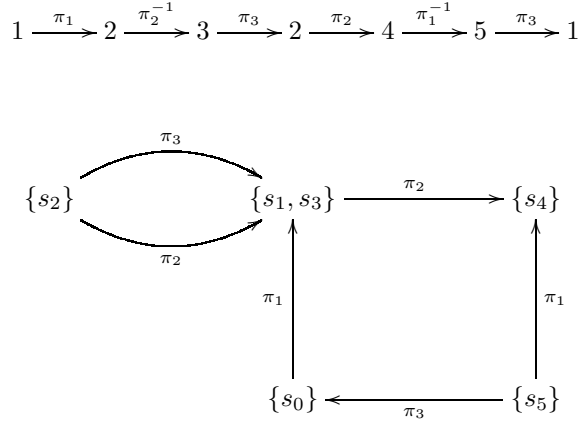

This new graph is called the category of the trail. See Figure \ref{fig:cnbwcateg} for an example. One can now inquire how many trails of a certain category appears in a given graph. The crucial property of categories is the difference between the number of edges and vertices. We introduce the new definition. 

\begin{defn}\label{defn:categcharacter}
For a category $\Gamma$ its characteristic will refer to the quantity $e-v$, where $e$ is the number of edges of $\Gamma$ and $v$ is the number of vertices. 
\end{defn}

The characteristic of a category determines how likely it is to appear in $G(n,2d)$. We will refer to the following bound from \cite[Lemma 14]{DJPP}. Let $X_\Gamma^{(n)}$ be the number of trails of category $\Gamma$ that appear in $G(n,2d)$ for $n > e$. Then 
\eq\label{eq:expectcateg}
\E\left( X_\Gamma^{(n)}  \right) \le \frac{1}{[n-k]_\gamma},
\en
where $\gamma$ is the characteristic of $\Gamma$ and $[x]_j$ is the falling factorial $[x]_j=x(x-1)\cdots(x-j+1)$.

\begin{lemma}\label{lem:nobadcharacter}
For any $S_0>0$ and any $\epsilon >0$, for all large enough $T$, the probability that there exists a bad $\mathrm{CNBW}$ of length at most $K$ in some $G(T,s)$, $0\le s\le S_0$, whose category has characteristic $2$ or higher is at most $\epsilon$. 
\end{lemma}

\begin{proof} For a given graph $G$ with law $G(n,2d)$, the probability that it has a bad CNBW of characteristic $2$ or higher is $O(n^{-2})$. This follows from \cite[Proposition 15]{DJPP} when we start the summation in the statement from $i=2$ (instead of $i=1$). The constant in the big-$O$ depends only on $d$ and $K$.   

Suppose at dimension $T$ the graphs have $n$ vertices. Given $S_0$, there are a Poisson with mean $nS_0$ many transpositions that occur during time $[0, S_0]$. Thus with probability at least $1- \frac{\e}{2},$  $O(n)$ many distinct graphs that ever exist during time $[0, S_0]$. By a union bound, the probability that any of them will have a bad CNBW of characteristic $2$ or higher is $O(1/n)$. 
We skip the easy details. 
\end{proof}

\begin{defn}\label{defn:tangle}
For any pair $(l,j)\in \NN^2$ we say that a graph is $(l,j)$ tangle free if any two cycles of length at most $l$ are at a graph distance at least $j$ from one another. 
\end{defn}

\begin{rmk}
The concept of \textit{tangle-free-ness} has been given other similar but slightly different definitions. See, for example, \cite{friedmanalon}. Our definition is relevant only to the following argument. 
\end{rmk}

\begin{lemma}\label{lem:tanglecnbw}
If a graph is $(l,l)$ tangle free then it has no bad CNBW of length $l$ or less.
\end{lemma}

\begin{proof} Consider a trail corresponding to a bad CNBW of length $l$ (or less) by starting from a vertex and walking along the walk. The lemma is clear if one imagines the category graph for this trail. 

This graph has to have a cycle, since otherwise it is not a cyclically non backtracking walk. However, the graph cannot be the cycle itself, since, otherwise the CNBW will be a repeated cycle which cannot be bad. Thus, consider the first cycle one encounters following the trail. That is, the first time we encounter a vertex encountered before. If one \textit{erases} this cycle from the graph, the remaining is non-empty. By erasing a cycle we mean the following procedure: consider the part of the trail that constitutes the cycle. Remove all vertices and edges, except the first vertex where the cycle starts and the edge that connects the final vertex (which is the same as the first) to the next vertex (which is not in the cycle).

Now, we claim that the remaining graph must also a have a cycle. This is evident because erasing the cycle did not destroy the property that the remaining walk is closed. The two cycles thus found have length at most $l$ and a graph distance at most $l-2$. If it exists in the graph, it contradicts the $(l,l)$ tangle free property. 
\end{proof}

The plan of the proof of Proposition \ref{prop:nobadwalk} is the following. Consider a bad CNBW of length at most $K$ that exists at some point in the rectangle $R_0$. Consider its trail and the corresponding category graph. Fix the time axis, and move forward in dimension. As a new vertex gets added to the graph, it is possible that it will have edges with the existing vertices. 
Suppose the new vertex is $k$. According to the rule of CRT, an existing edge 
\[
i \stackrel{\pi_i}{\longrightarrow} j \quad \text{or} \quad i \stackrel{\pi^{-1}_i}{\longrightarrow} j 
\]
can only change into 
\eq\label{eq:edgedouble}
i \stackrel{\pi_i}{\longrightarrow} k \stackrel{\pi_i}{\longrightarrow} j \quad \text{or} \quad i \stackrel{\pi^{-1}_i}{\longrightarrow} k \stackrel{\pi^{-1}_i}{\longrightarrow} j, \quad \text{respectively}.  
\en
Introduce the new vertex into the walk by including it as above by choosing one of the possible edges it belongs to. 

Modify the category graph by adding the new vertex and the new edge. The important observation is that the characteristic of the category of the trail is non-decreasing. This is because when a new vertex gets added to the existing trail a new edge also gets added, and therefore the trail remains a bad CNBW.  

Thus, for a bad CNBW $\omega$ of characteristic $\gamma$, there must exist another bad CNBW $\omega'$ of characteristic at least $\gamma$ at dimension $T$ such that the former can be obtained from the latter by deletion of vertices. This is a very similar situation to the arguments in Subsection \ref{sec:halvingconv} where we looked into larger cycles at dimension $T$ shrinking to smaller cycles at a lower dimension. In short, one can expect an $O(1)$ many vertices to be included in the bad CNBW as the dimension increases to $T$. Suppose the length of $\omega$ is $j$ and that of $\omega'$ is $l$. There are $i=l-j$ many new vertices added to the trail. The proof of Lemma \ref{lem:tanglecnbw} shows that there must be at least two cycles in $\omega$ that are at distance at most $j-2$ from each other.

Some of the $i$ many extra vertices will possibly increase the size of these cycles, some will possibly increase the graph distance. In any case, at dimension $T$ we will have two cycles of length at most $l$ that are at a graph distance at most $l-2$. Thus, the $(l,l)$ tangle free condition will be violated at some time point at dimension $T$.

Our proof has two parts. One, to show that for any fixed $l$, with high probability, the $(l,l)$ tangle free condition holds at dimension $T$ during time $[0, S_0]$. Two, to show that for any $k\in \NN$, there is a large $l\in \NN$, such that with high probability all bad CNBWs at a lower dimension of length at most $k$ have shrunk from bad CNBWs of length at most $l$ at dimension $T$, for all large enough $T$. 

\begin{lemma}\label{lem:tanglefreedimt}
Fix any $\epsilon>0$ and any $L\in \NN$. Then, for all large enough $T$, we have 
\[
\P\left(  G(T,s) \; \text{is $(L,L)$ tangle free for all $s\in[0, S_0]$ }   \right) \ge 1- \epsilon. 
\]
\end{lemma}
The above statement is essentially the same as modifying the definition of
 the stopping time $\tau_1$ in \eqref{eq:defnst} to be the time till which all cycles of length at most $L$ have disjoint $L-$ neighborhoods. The proof hence is essentially the same as the proof of Lemma \ref{lem:stimeinf}. 
We include the proof for clarity and to introduce some new structures, needed for the proof of the next lemma.  

  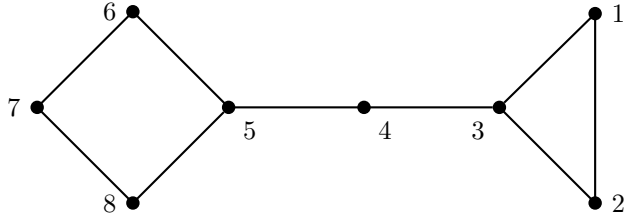
\begin{figure}[t]
      \begin{center}
        \begin{tikzpicture}[scale=1.8,vert/.style={circle,fill,inner sep=0,
              minimum size=0.15cm,draw},>=stealth]
            \draw[thick] (0,0) node[vert,label=below left:3](s0) {}
                  -- ++ (0.707cm, -0.707cm) node[vert,label=right:2](s1) {}
                  -- ++ (0,1.4) node[vert,label=right:1](s2) {} 
                  -- (s0) {}
                  -- ++(-1, 0 cm) node[vert,label=below right:4](s4) {}
                   -- ++(-1, 0 cm) node[vert,label=below right:5](s5) {}
                  -- ++(-0.707cm,0.707cm) node[vert,label=left:6](s6) {}
                  -- ++(-0.707cm,-0.707cm) node[vert,label=left:7](s7) {}
                  -- ++(0.707cm,-0.707cm) node[vert,label=left:8](s8) {}
                  --(s5);
        \end{tikzpicture}
      \end{center}
      \caption{The trail $1\rightarrow 2\rightarrow 3\rightarrow 4\rightarrow 5\rightarrow 6\rightarrow 7\rightarrow 8\rightarrow 5\rightarrow 4 \rightarrow 3\rightarrow2\rightarrow 1$ is a double lollipop. The edge labels are not shown.}
      \label{fig:dlolly}
    \end{figure} 

\begin{proof} Consider any three integers $i,j,k$, each less than $L$. Consider a category graph that connects two cycles of sizes $i$ and $j$ by a single path of length $k\ge 0$. We call this graph (or the trail) a \textit{double-lollipop} for obvious reason. See Figure \ref{fig:dlolly} for a case of $i=4,j=3,k=1$. 
Given a double-lollipop associate a word with it by choosing any hamiltonian path and traversing the edges along it.  
The importance of the double-lollipops come from the Lemmas \ref{lem:nobadcharacter} and \ref{lem:tanglecnbw}. A graph is not tangle-free if and only if either (i) two short cycles intersect at more than one vertex, or (ii) there are two short cycles joined by a short path (i.e., a double-lollipop). Case (i) is unlikely by Lemma \ref{lem:nobadcharacter}. It remains is to argue that case (ii) is also unlikely.   

It is known from \cite[Section 5]{DJPP} that, for $T$ large enough, the probability that $G(T,0)$ is not $(L,L)$ tangle free is at most $\epsilon/2$. So, we restrict attention to new double-lollipops that are ever born during $(0, S_0]$. The argument is similar to that of pre-cycle counts in Lemma \ref{lem:precyclecount}, Lemma \ref{lem:birthrate}, and the proof of Lemma \ref{lem:stimeinf}. The following is the idea. 

Suppose we are given a trail corresponding to a double-lollipop with word $w$. Then the length of $w$ is $i+j+k\le 3L$. We define a pre-lollipop to be a trail (or a collection of two trails) that is obtained by pre-multiplying a transposition with the trail of the double-lollipop. Suppose we show, with high probability, that there is $O(\log n)$ many pre-lollipops that ever exist during time $[0, S_0]$, independent of $n$. Then, only $O(\log n)$ many vertices are included in these pre-lollipops. Thus, if a double-lollipop ever has to occur during time $[0, S_0]$, the required transposition must involve two vertices from these pre-lollipops. But, this is highly unlikely, since each transposition occurs with probability $2/n^2$ and only a Poisson with mean $n S_0$ ever transpositions occur during $[0, S_0]$.

  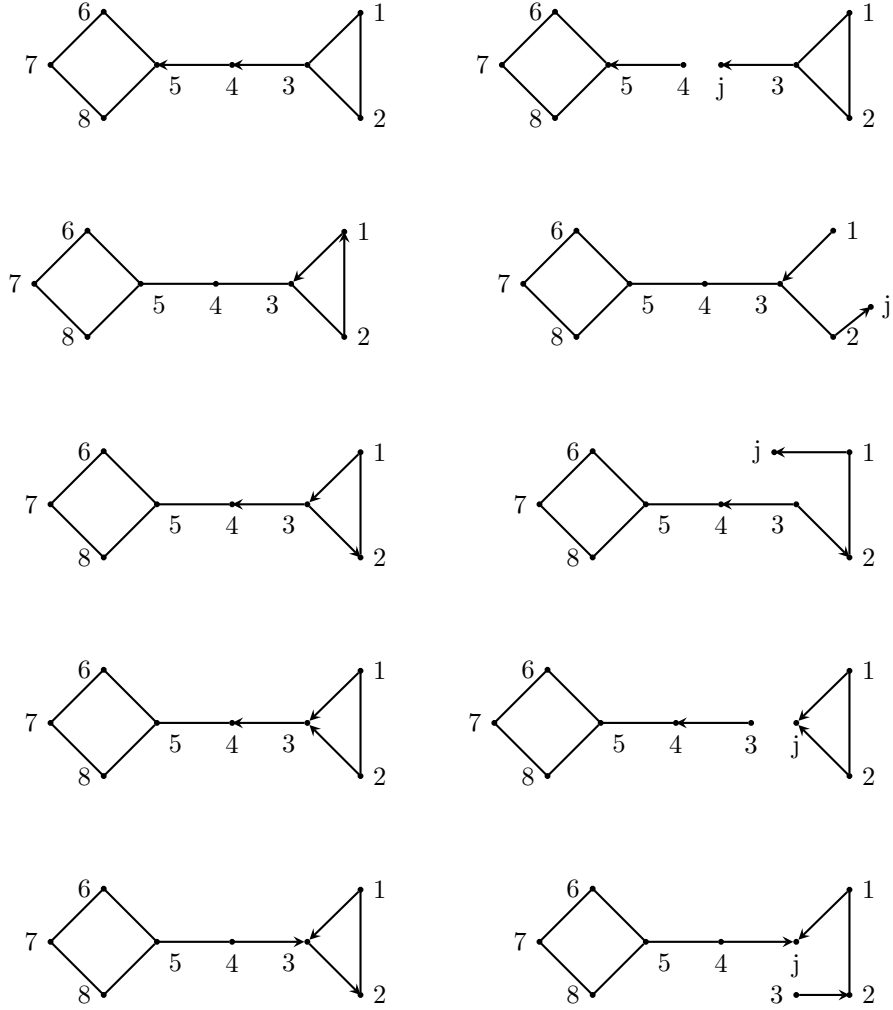
\begin{figure}[t]
      \begin{center}
       \begin{tikzpicture}[scale=1.0,vert/.style={circle,fill,inner sep=0,
              minimum size=0.05cm,draw},>=stealth]
            \draw[thick] (0,0) node[vert,label=below left:3](s0) {}
                  -- ++ (0.707cm, -0.707cm) node[vert,label=right:2](s1) {}
                  -- ++ (0,1.4) node[vert,label=right:1](s2) {} 
                  -- (s0) {};
                  \draw[thick, ->](s0) -- ++(-1, 0 cm) node[vert,label=below:4](s4) {};
                  
                 \draw[thick, ->] (s4) -- ++(-1, 0 cm) node[vert,label=below right:5](s5) {};

                 \draw[thick] (s5) -- ++(-0.707cm,0.707cm) node[vert,label=left:6](s6) {}
                  -- ++(-0.707cm,-0.707cm) node[vert,label=left:7](s7) {}
                  -- ++(0.707cm,-0.707cm) node[vert,label=left:8](s8) {}
                  --(s5);

                 \pgftransformxshift{6.5cm} 
                  \draw[thick] (0,0) node[vert,label=below left:3](s0) {}
                  -- ++ (0.707cm, -0.707cm) node[vert,label=right:2](s1) {}
                  -- ++ (0,1.4) node[vert,label=right:1](s2) {} 
                  -- (s0) {};
                  \draw[thick, ->](s0) -- ++(-1, 0 cm) node[vert,label=below:j](s4) {};
                  
                  \draw[thick] (-1.5,0) node[vert,label=below:4](s9) {};
                 \draw[thick, ->] (s9) -- ++(-1, 0 cm) node[vert,label=below right:5](s5) {};

                 \draw[thick] (s5) -- ++(-0.707cm,0.707cm) node[vert,label=left:6](s6) {}
                  -- ++(-0.707cm,-0.707cm) node[vert,label=left:7](s7) {}
                  -- ++(0.707cm,-0.707cm) node[vert,label=left:8](s8) {}
                  --(s5);

        \end{tikzpicture}
        
        \vspace{1cm}
        
       \begin{tikzpicture}[scale=1.0,vert/.style={circle,fill,inner sep=0,
              minimum size=0.05cm,draw},>=stealth]
            \draw[thick] (0,0) node[vert,label=below left:3](s0) {}
                  -- ++ (0.707cm, -0.707cm) node[vert,label=right:2](s1) {};
                  \draw[thick,->] (s1)-- ++ (0,1.4) node[vert,label=right:1](s2) {}; 
                  \draw[thick,->] (s2) -- (s0) {};
                  \draw[thick] (s0) -- ++(-1, 0 cm) node[vert,label=below:4](s4) {}
                  
              -- ++(-1, 0 cm) node[vert,label=below right:5](s5) {}

                -- ++(-0.707cm,0.707cm) node[vert,label=left:6](s6) {}
                  -- ++(-0.707cm,-0.707cm) node[vert,label=left:7](s7) {}
                  -- ++(0.707cm,-0.707cm) node[vert,label=left:8](s8) {}
                  --(s5);

                 \pgftransformxshift{6.5cm} 
                 \draw[thick] (0,0) node[vert,label=below left:3](s0) {}
                  -- ++ (0.707cm, -0.707cm) node[vert,label=right:2](s1) {};
                  \draw[thick,->] (s1)-- ++ (.5,.4) node[vert,label=right:j](s2) {}; 
                  
                  \draw[thick] (0.707cm,0.707cm) node[vert, label=right:1](s9) {};
                  \draw[thick,->] (s9) -- (s0) {};
                  \draw[thick] (s0) -- ++(-1, 0 cm) node[vert,label=below:4](s4) {}
                  
              -- ++(-1, 0 cm) node[vert,label=below right:5](s5) {}

                -- ++(-0.707cm,0.707cm) node[vert,label=left:6](s6) {}
                  -- ++(-0.707cm,-0.707cm) node[vert,label=left:7](s7) {}
                  -- ++(0.707cm,-0.707cm) node[vert,label=left:8](s8) {}
                  --(s5);

        \end{tikzpicture}
  
        \vspace{1cm}
        
        \begin{tikzpicture}[scale=1.0,vert/.style={circle,fill,inner sep=0,
              minimum size=0.05cm,draw},>=stealth]
            \draw[thick] (0,0) node[vert,label=below left:3](s0) {};
                  \draw[thick,->] (s0)-- ++ (0.707cm, -0.707cm) node[vert,label=right:2](s1) {};
                  \draw[thick] (s1)-- ++ (0,1.4) node[vert,label=right:1](s2) {}; 
                  \draw[thick,->] (s2) -- (s0) {};
                  \draw[thick, ->] (s0) -- ++(-1, 0 cm) node[vert,label=below:4](s4) {};
                  
                 \draw[thick] (s4)-- ++(-1, 0 cm) node[vert,label=below right:5](s5) {}

                -- ++(-0.707cm,0.707cm) node[vert,label=left:6](s6) {}
                  -- ++(-0.707cm,-0.707cm) node[vert,label=left:7](s7) {}
                  -- ++(0.707cm,-0.707cm) node[vert,label=left:8](s8) {}
                  --(s5);

                 \pgftransformxshift{6.5cm} 
             \draw[thick] (0,0) node[vert,label=below left:3](s0) {};
                  \draw[thick,->] (s0)-- ++ (0.707cm, -0.707cm) node[vert,label=right:2](s1) {};
                  \draw[thick] (s1)-- ++ (0,1.4) node[vert,label=right:1](s2) {};
                  \draw[thick,->] (s2)--++(-1,0) node[vert,label=left:j](s9) {};
                  \draw[thick,->] (s0) -- ++(-1, 0 cm) node[vert,label=below:4](s4) {};
                  
                  \draw[thick] (s4)-- ++(-1, 0 cm) node[vert,label=below right:5](s5) {}

                -- ++(-0.707cm,0.707cm) node[vert,label=left:6](s6) {}
                  -- ++(-0.707cm,-0.707cm) node[vert,label=left:7](s7) {}
                  -- ++(0.707cm,-0.707cm) node[vert,label=left:8](s8) {}
                  --(s5);
        \end{tikzpicture}

        \vspace{1cm}
        
           \begin{tikzpicture}[scale=1.0,vert/.style={circle,fill,inner sep=0,
              minimum size=0.05cm,draw},>=stealth]
            \draw[thick] (0,0) node[vert,label=below left:3](s0) {};
                  \draw[thick,<-] (s0)-- ++ (0.707cm, -0.707cm) node[vert,label=right:2](s1) {};
                  \draw[thick] (s1)-- ++ (0,1.4) node[vert,label=right:1](s2) {}; 
                  \draw[thick,->] (s2) -- (s0) {};
                  \draw[thick, ->] (s0) -- ++(-1, 0 cm) node[vert,label=below:4](s4) {};
                  
                 \draw[thick] (s4)-- ++(-1, 0 cm) node[vert,label=below right:5](s5) {}

                -- ++(-0.707cm,0.707cm) node[vert,label=left:6](s6) {}
                  -- ++(-0.707cm,-0.707cm) node[vert,label=left:7](s7) {}
                  -- ++(0.707cm,-0.707cm) node[vert,label=left:8](s8) {}
                  --(s5);

                 \pgftransformxshift{6.5cm} 
    \draw[thick] (0,0) node[vert,label=below:j](s0) {};
                  \draw[thick,<-] (s0)-- ++ (0.707cm, -0.707cm) node[vert,label=right:2](s1) {};
                  \draw[thick] (s1)-- ++ (0,1.4) node[vert,label=right:1](s2) {}; 
                  \draw[thick,->] (s2) -- (s0) {};
                 
                 \draw[thick] (-.6,0) node[vert,label=below:3] (s9) {};
                  \draw[thick, ->] (s9) -- ++(-1, 0 cm) node[vert,label=below:4](s4) {};
                  
                 \draw[thick] (s4)-- ++(-1, 0 cm) node[vert,label=below right:5](s5) {}

                -- ++(-0.707cm,0.707cm) node[vert,label=left:6](s6) {}
                  -- ++(-0.707cm,-0.707cm) node[vert,label=left:7](s7) {}
                  -- ++(0.707cm,-0.707cm) node[vert,label=left:8](s8) {}
                  --(s5);
                \end{tikzpicture}

	\vspace{1cm}
	
          \begin{tikzpicture}[scale=1.0,vert/.style={circle,fill,inner sep=0,
              minimum size=0.05cm,draw},>=stealth]
            \draw[thick] (0,0) node[vert,label=below left:3](s0) {};
                  \draw[thick,->] (s0)-- ++ (0.707cm, -0.707cm) node[vert,label=right:2](s1) {};
                  \draw[thick] (s1)-- ++ (0,1.4) node[vert,label=right:1](s2) {}; 
                  \draw[thick,->] (s2) -- (s0) {};
                  \draw[thick, <-] (s0) -- ++(-1, 0 cm) node[vert,label=below:4](s4) {};
                  
                 \draw[thick] (s4)-- ++(-1, 0 cm) node[vert,label=below right:5](s5) {}

                -- ++(-0.707cm,0.707cm) node[vert,label=left:6](s6) {}
                  -- ++(-0.707cm,-0.707cm) node[vert,label=left:7](s7) {}
                  -- ++(0.707cm,-0.707cm) node[vert,label=left:8](s8) {}
                  --(s5);

                 \pgftransformxshift{6.5cm} 
                \draw[thick] (0,0) node[vert,label=below:j](s0) {};
                \draw[thick] (0,-.707cm) node[vert,label=left:3] (s9) {};
                  \draw[thick,->] (s9) -- (0.707cm, -0.707cm) node[vert,label=right:2](s1) {};
                  \draw[thick] (s1)-- ++ (0,1.4) node[vert,label=right:1](s2) {}; 
                  \draw[thick,->] (s2) -- (s0) {};
                  \draw[thick, <-] (s0) -- ++(-1, 0 cm) node[vert,label=below:4](s4) {};
                  
                 \draw[thick] (s4)-- ++(-1, 0 cm) node[vert,label=below right:5](s5) {}

                -- ++(-0.707cm,0.707cm) node[vert,label=left:6](s6) {}
                  -- ++(-0.707cm,-0.707cm) node[vert,label=left:7](s7) {}
                  -- ++(0.707cm,-0.707cm) node[vert,label=left:8](s8) {}
                  --(s5);
                \end{tikzpicture}

      \end{center}
      \caption{Distinct classes of pre-lollipops for the double-lollipop in Figure \ref{fig:dlolly}. Arrows represent the signs of the letters on the edges, and are only shown on the edges incident on the vertex that is relevant for the transposition. The vertex $j$ refers to the other label in the transposition and is assumed to be not a vertex in the walk itself. In each case the resulting collection of trails is contained in the neighborhoods of one or more cycles. In all other cases either the graph does not change, except for vertex labels, or the pre-lollypop is similar to one above by symmetry.}
      \label{fig:prelolly}
    \end{figure}

Consider the stopping time $\stime$ similar to Lemma \ref{lem:stimeinf}, except in \eqref{eq:defnst} we consider short cycles to be those with size at most $L$. We start by showing that the number of pre-lollipops born during $[0, \stime)$ is of order $\log n$. Fix a double-lollipop and imagine the different pre-lollipops that are possible and transpositions $[x,y]$ which convert a pre-lollipop to the double lollipop. 
There are two possible cases: $x$ can be either a vertex of a cycle in the double-lollipop, or a vertex in the path connecting two cycles.

In the first case, the resulting graph is contained in the $L$ neighborhood of the other cycle. In the second case, we might get two subgraphs each contained in the $L$-neighborhood of one of the cycles. See Figure \ref{fig:prelolly} for a list of possibilities. Until $\stime$ there are at most $O\left(\sqrt{\log n}\right)$ many cycles of length at most $L$ that ever exist during $[0, S_0]$ with high probability. Thus, the total number of vertices contained in the $L$ neighborhood of these cycles is also $O\left( \sqrt{\log n} \right)$, where the big-$O$ involves a constant that depends on $d$ and $L$. Since pre-lollipops can only get born if the transposition involves vertices from these neighborhoods, there can be at most $O(\log n)$ many such pre-lollipops that ever exists during $[0, \stime)$ with high probability. The rest of the argument is taken care of as in Lemma \ref{lem:stimeinf} and the discussion in the previous paragraph. 
\end{proof}

\begin{lemma}\label{lem:badcnbwshrink} For any $T_0,S_0>0,$ $K\in\NN$ and any $\epsilon >0$, for all $T$ large enough we have
\[
\begin{split}
 \P&\left(  G(t,s)\; \text{is $(K,K)$ tangle free for all $(t,s)\in R_0$}\right) \ge 1-\e ,
\end{split}
\]
where 
$R_0:=[T-T_0, T] \times [0, S_0].$
\end{lemma}
\begin{proof} This proof is very similar to the arguments in Section \ref{sec:halvingconv}. 
If some $G(t,s)$ is not $(K,K)$ tangle free then there exists two cycles $\mathcal{C}_1$ and $\mathcal{C}_2$ of size at most $K$ which are at distance at most $K$. Moving forward in dimension, these cycles grow into an overlap of cycles $\tilde {\mathcal{C}}_1$,$\tilde {\mathcal{C}}_2$.
  By \eqref{eq:thirdunion} there exists $L$ such that with probability at least $1-\e/8$ neither $\tilde {\mathcal{C}}_1$ or $\tilde {\mathcal{C}}_2$ contains a cycle of size bigger than $L$. Now Lemma \ref{lem:tanglefreedimt} then implies that with probability at least $1-\e/8$, $\tilde {\mathcal{C}}_1$ and $\tilde {\mathcal{C}}_2$ should contain exactly one cycle each and also should be disjoint. This is possible only if $C_1$ and $C_2$ were disjoint and had a path $\mathcal{P}$ joining them of length at most $K$ in $G(t,s).$ Now growing in dimension $\mathcal{P}$ will grow in to $\tilde{\mathcal{P}}$. Now some of the vertices on $\tilde{\mathcal{P}}$ might be repeated and also be the same as some of the vertices on either $\tilde{ \mathcal{C}}_1$ or $\tilde{ \mathcal{C}}_2$. However it is clear that one gets a double-lollipop at dimension $T$ with two cycles $\tilde {\mathcal{C}}_2$ and $\tilde{ \mathcal{C}}_2$ and a path $ {\mathcal{P}}_*$ joining them of length say $\ell.$ Since $\mathcal{P}$ had length at most $K$ it is clear that ${\mathcal{P}}_*$ has at least $\ell-K$ shrinking vertices. Also both 
 $\tilde{\mathcal{C}}_1$ and $\tilde{\mathcal{C}}_2$ have at least $|\tilde{\mathcal{C}}_1|-K$ and $|\tilde {\mathcal{C}}_2|-K$ shrinking vertices respectively. Again by Lemma \ref{lem:tanglefreedimt} with probability at least $1-\e/8$, we must have $\ell >L$. 
 
Thus, we have shown that
 \[
 \begin{split}
P&\left( \text{$G(t,s)$ is not $(K,K)$ tangle free}\right) - \e/2 \\
&\le P\left( \text{$\exists\;$ a lollipop of length $\ell>L$  with at 
least $\ell-3K$ shrinking vertices} \right).
\end{split}
\]
Define the size of double-lollipop to be the length of the word associated to it. 


Clearly now it suffices to show that there exists  an $L$, such that the probability that any of the possible double-lollipops of size $l > L$ shrinks to a double-lollipop of size at most $K$ by dimension $T-T_0$ is bounded above by $\epsilon/4$. 

Since the argument is very similar to that in Section \ref{sec:halvingconv}, we simply point out the major differences. Consider \eqref{eq:firstunion}. We will modify it so that $E_T(L)$ and $F(w,I)$ will refer to double-lollipops and not cycles. In \eqref{eq:secondunion}, we modify $G_w$ to refer to the number of double-lollipops that ever get created at dimension $T$ during time $[0,S_0]$. Now similar arguments as in Lemma \ref{lem:howmnycyc} gives us that $$\E(G_w)\le c_0 S_0 |w|,$$ where $c_0$ is some absolute constant. 
Equation \eqref{candidate} adapted to this setting now shows that the number of possible $w's$ of length $l$ which can shrink to a double-lollipop of size $k$ is at most $l^k(2d)^k.$ 

Thus by application of union bound similar to \eqref{eq:thirdunion} we are done. 
\end{proof}

\begin{proof}[Proof of Proposition \ref{prop:nobadwalk}] Follows from Lemmas \ref{lem:tanglecnbw} and \ref{lem:badcnbwshrink}.
\end{proof}
\begin{proof}[Proof of Theorem \ref{prop:eigenvalues}] This is almost identical to the proof of Theorem 5 in \cite{JP14}. We provide the sketch below.  Since by Proposition \ref{prop:nobadwalk} with probability approaching one, all CNBW's are repeated cycles, their weak limits can be computed from one another. Formally for any positive integer $K$, as $T \to \infty,$
\begin{equation}\label{wl1}
\{\mathrm{CNBW}_{k}(T+\cdot,\cdot):1\le k\le K\}\stackrel{weakly}{\rightarrow} \{\sum_{j\mid k} 2j N_j(\cdot,\cdot):1\le k\le K\}.
\end{equation}
The above follows from Theorem \ref{prop:weakconv}, Proposition \ref{prop:nobadwalk}, and  easy applications of Continuous Mapping Theorem and  Slutsky's theorem.
  Using  \eqref{eq:eigencnbw} and \eqref{wl1}  the polynomials $f_k$ are obtained from $\Gamma_k$ by expressing $N_k(\cdot,\cdot)$ as linear combinations of  $\sum_{j\mid \ell} 2j N_j(\cdot,\cdot).$ This is done using the M\"obius inversion formula. The inversion formula and  
the polynomial basis referred to is explicitly evaluated in \cite[eqn. (16)]{JP14}. 
\end{proof}

\begin{proof}[Proof of Theorem \ref{prop:chebycov}]
 By Proposition \ref{prop:nobadwalk}
\[
2 \tr T_k\left( G(\infty+t, s)  \right)= (2d-1)^{-k/2} \sum_{j\mid k} 2j N_j(t,s). 
\]
Taking expectations on both sides, we get
\[
2 \E \left[  \tr T_k\left( G(\infty+t, s)  \right)   \right]= (2d-1)^{-k/2} \sum_{j\mid k} a(d,j). 
\]

As $d$ tends to infinity, from Theorem \ref{prop:covariance}, only the $N_k$ term in the difference survives. Thus, the field \eqref{eq:chebyfield} has the same weak limit as one half of the field \eqref{eq:scaledfield}, which is $U/2$. This completes the proof of the proposition.  
\end{proof}

\bibliographystyle{alpha}
\bibliography{GFF}

\end{document}